\documentclass[12pt]{amsart}
\usepackage{amssymb}
\title{Cacti and cells}
\author{Ivan Losev}
\address{Department
of Mathematics, Northeastern University, Boston MA 02115 USA}
\email{i.loseu@neu.edu}
\newcommand{\C}{\mathbb{C}}

\newcommand{\Z}{\mathbb{Z}}
\newcommand{\Hom}{\operatorname{Hom}}
\newcommand{\g}{\mathfrak{g}}

\newcommand{\Hecke}{\mathcal{H}}
\newcommand{\bc}{\mathbf{c}}
\newcommand{\bA}{\bar{A}}
\newcommand{\h}{\mathfrak{h}}
\newcommand{\bo}{\mathfrak{b}}
\newcommand{\Cact}{\mathsf{Cact}}
\newcommand{\U}{\mathcal{U}}
\newcommand{\OCat}{\mathcal{O}}
\newcommand{\HC}{\operatorname{HC}}
\newcommand{\Loc}{\operatorname{Loc}}
\newcommand{\WC}{\mathfrak{WC}}
\newcommand{\wc}{\mathfrak{wc}}
\newcommand{\B}{\mathcal{B}}
\newcommand{\Orb}{\mathbb{O}}
\newcommand{\Walg}{\mathcal{W}}
\newcommand{\VA}{\operatorname{V}}

\newcommand{\J}{\mathcal{J}}

\newcommand{\m}{\mathfrak{m}}
\newcommand{\gr}{\operatorname{gr}}
\newcommand{\Wf}{\mathfrak{W}}
\newcommand{\Coh}{\operatorname{Coh}}
\newcommand{\param}{\mathfrak{P}}
\newcommand{\Cat}{\mathcal{C}}
\newcommand{\z}{\mathfrak{z}}
\newcommand{\I}{\mathcal{I}}
\newtheorem{Thm}{Theorem}[section]
\newtheorem{Prop}[Thm]{Proposition}
\newtheorem{Cor}[Thm]{Corollary}
\newtheorem{Lem}[Thm]{Lemma}
\theoremstyle{definition}
\newtheorem{Ex}[Thm]{Example}

\numberwithin{equation}{section}
\begin{document}
\begin{abstract}
The goal of this paper is to construct an action of the cactus group of a Weyl group
$W$ on $W$ that is nicely compatible with Kazhdan-Lusztig cells. The action is realized
by the wall-crossing bijections that are combinatorial shadows of wall-crossing functors
on the category $\mathcal{O}$.
\end{abstract}
\maketitle
\section{Introduction}
\subsection{Cells}
Let $G$ be an adjoint semisimple algebraic group over $\C$, $\g$ be its Lie algebra
and $W$ be the Weyl group. Introduce an independent variable $v$ and consider the
(equal-parameter) Hecke algebra $\Hecke_v$. Recall that it is defined as follows.
It has a basis $T_w$ with $w\in W$ and the multiplication is recovered from
$$T_sT_w=\begin{cases}&T_{sw},\quad \text{if } \ell(sw)>\ell(w),\\
&(v-v^{-1})T_{w}+T_{sw},\quad \text{else}.\end{cases}$$
The algebra $\Hecke_v$ admits another remarkable basis, the Kazhdan-Lusztig
basis $C_w, w\in W$. Using this basis one can introduce the left, right and two-sided
pre-orders $\preceq_{L},\preceq_{R},\preceq_{LR}$. For example, we write
$w\preceq_Lw'$ if $C_w$ belongs to the based left ideal generated by $C_{w'}$.
Equivalence classes for these pre-orders are called left, right and two-sided
cells. By the construction, the inversion $w\mapsto w^{-1}$ maps left cells to right cells and vice-versa,
and preserves the two-sided cells. The cells are of great importance for the representation theory of
the universal enveloping algebra $U(\g)$ and also for that of the forms
$G(\mathbb{F}_q)$ of $G$ over the finite fields.

To cells one can assign different invariants. For example, to a two-sided
cell $\bc$ Lusztig assigned a finite group $\bA_{\bc}$, see \cite[Section 13.1]{Lusztig_orange}.
To every left cell $\sigma\subset \bc$ he assigned a subgroup $H_\sigma\subset \bA_{\bc}$,
see \cite{Lusztig_leading}, defined up to conjugacy.

To finish our discussion of cells let us explain what happens when $G=\operatorname{PGL}_n(\C)$.
Here the two-sided cells are in one-to-one correspondence with the partitions $\lambda$
of $n$. The left cells inside of the two-sided cell corresponding to $\lambda$ are in
one-to-one correspondence with the standard Young tableaux of shape $\lambda$.
The left and right cells containing $w\in W=\mathfrak{S}_n$ are determined from
the RSK correspondence that maps $w$ to a pair of standard Young tableaux of the same
shape. These Young tableaux parameterize the left and the right cells containing $w$.
The Lusztig group $\bA_{\bc}$ is trivial for any two-sided cell $\bc$.
\subsection{Cacti}
Now let us recall the cactus group $\Cact_W$. Let $D$ denote the Dynkin diagram of $W$. The group
$\Cact_W$ is generated by elements $\tau_{D_1}$, where $D_1\subset D$ runs over all
connected subdiagrams of $D$. The relations are as follows
\begin{equation}\label{eq:cactus_rel}
\begin{split}
&\tau_{D_1}^2=1,\\
&\tau_{D_1}\tau_{D_2}=\tau_{D_2}\tau_{D_1},\quad \text{if } D_1\cup D_2\text{ is disconnected},\\
&\tau_{D_1}\tau_{D_2}=\tau_{D_2}\tau_{D_1^*},\quad \text{if }D_1\subset D_2,
\end{split}
\end{equation}
where $D_1^*$ is obtained from $D_1$ by the involution of $D_2$ induced by the longest
element $w_{D_2}$ of the parabolic subgroup $W_{D_2}\subset W$ corresponding to $D_2$.

Below we will often write $\Cact_D$ instead of $\Cact_W$.

According to \cite[Theorem 4.7.2]{DJS},  the group $\Cact_W$ is the orbifold fundamental group
of the real locus in the stack $\overline{\mathbb{P}\h^{reg}}/W$, where
$\overline{\mathbb{P}\h^{reg}}$ is the wonderful compactification of
$\mathbb{P}\h^{reg}:=\h^{reg}/\C^\times$, where $\h^{reg}$ is the regular locus
in the reflection representation $\h$ of $W$.

The cactus group $\Cact_{W}$ should be thought as a ``crystal limit'' of
the braid group $\mathsf{Br}_W$ of $W$ as justified, for example,
by \cite{HK}. Note that, similarly to the braid group, $\Cact_{W}$ admits
an epimorphism onto $W$ given by
$\tau_{D_1}\mapsto w_{D_1}$.

\subsection{Main result and motivations}
Here is the main result of the present paper.

\begin{Thm}\label{Thm:cactus}
There is an action of $\Cact_W\times \Cact_W$ on $W$ having the following properties.
\begin{itemize}
\item[(i)] The first copy of $\Cact_W$ preserves right cells and permutes left cells
preserving the Lusztig subgroups.
\item[(ii)] The map $w\mapsto w^{-1}$ switches the actions of the first and the
second copy.
\item[(iii)] Let $w\in W$  decompose as $w=w'w''$, where $w''\in W_{D_1}$
and $w'$ is shortest in its right $W_{D_1}$-coset. Then $\tau_D(w'w'')=w'\tau^{(D_1)}_{D_1}(w'')$,
where $\tau_{D_1}^{(D_1)}$ is the element corresponding to $D_1$ in $\Cact_{D_1}$.
\item[(iv)] In type $A$, the orbits of the first copy of $\Cact_W$ are precisely the right cells.
\end{itemize}
\end{Thm}
Note that outside type A we always have two left cells inside of a single two-sided cell
with different Lusztig subgroups. So (i) implies that an analog of (iv) fails outside of
type $A$.

For us, there are two motivations for this theorem. One comes from a new (conjectural) approach to cells
due to Bonnafe and Rouquier, \cite{BR}. In that approach, the (left) cells are defined as orbits of a suitable
Galois group action on $W$. It is expected  that $\Cact_W$ admits a homomorphism into that Galois group
and so should act on $W$ preserving the left cells. Currently, it is a conjecture that the Bonnafe-Rouquier
construction of cells is equivalent to   the original one.

The other motivation comes from the work of the author and Bezrukavnikov, \cite{BL}. A principle
stated in Section 9 of that paper says that given a suitable braid group action on a category
(or more generally, a ``braid groupoid'' action on a collection of categories) one should be
able to take a crystal limit of this action and get an action of the corresponding cactus
group(oid). A technical tool for this is to show that the functors corresponding to longest
elements in parabolic subgroups are perverse equivalences. This is an approach that we use
in the present paper.

\subsection{Structure of the paper}
In Section \ref{S_cat} we will recall some generalities on highest weight categories and on
perverse equivalences. Section \ref{S_Ug_rep} deals with various facts from the representation
theory of universal enveloping algebras that we need. There we review various versions of
the category $\mathcal{O}$, the wall-crossing functors, and W-algebras. In Section \ref{S_cacti}
we prove Theorem \ref{Thm:cactus}. We first establish the perversity of wall-crossing
functors corresponding to longest elements in the standard parabolic subgroups. This allows
us to define the bijections that constitute the action of the first copy of $\Cact_W$.
Then we prove (iii) of Theorem \ref{Thm:cactus}. Next, we check the cactus relations.
After that we prove (i) and (ii) of Theorem \ref{Thm:cactus}. We finish by proving (iv).
Finally, in Section \ref{S_ramif} we briefly describe several potential ramifications
of our construction.

{\bf Acknowledgements}. I would like to thank Roman Bezrukavnikov, Ben Elias, Pavel Etingof, Iain Gordon,
Joel Kamnitzer, Victor Ostrik, Raphael Rouquier and Noah White for stimulating discussions. This work was
partially supported by the NSF under grants DMS-1161584, DMS-1501558.

\section{Preliminaries on categories and functors}\label{S_cat}
\subsection{Highest weight categories}
Let $\Cat$ be an abelian category equivalent to the category of finite dimensional modules
over a  finite dimensional algebra over $\C$. Let us write $\mathcal{T}$ for the set of irreducible
objects in $\Cat$. For $\tau\in \mathcal{T}$, we write $L(\tau)$ for the corresponding
simple object and $P(\tau)$ for its projective cover.

An additional structure of a highest weight category on $\Cat$ is a partial order
$\leqslant$ on $\mathcal{T}$ subject to a condition explained below. For
$\tau\in \mathcal{T}$, let $\Cat_{\leqslant \tau}$ denote the Serre span
of the simple objects $L(\tau')$ with $\tau'\leqslant \tau$. We write
$\Delta(\tau)$ for the projective cover of $L(\tau)$ in $\Cat_{\leqslant \tau}$.
The objects $\Delta(\tau), \tau\in \mathcal{T},$ are called {\it standard}.
The condition on $\leqslant$ is that the kernel of the natural epimorphism
$P(\tau)\twoheadrightarrow \Delta(\tau)$ is filtered by $\Delta(\tau')$
with $\tau'>\tau$.

The category $\mathcal{C}^{opp}$ is highest weight with the same order on $\mathcal{T}$.
The standard objects for $\Cat^{opp}$ are denoted by $\nabla(\tau)$ are called
{\it costandard}. An object of $\Cat^{opp}$ is called {\it tilting} if it is both standardly
filtered and costandardly filtered. Indecomposable tilting objects are labelled
by $\mathcal{T}$: there is a unique indecomposable tilting $T(\tau)$ that admits
a monomorphism from $\Delta(\tau)$ such that the cokernel is filtered by $\Delta(\tau')$
with $\tau'<\tau$.

Now let $\mathcal{T}'$ be a poset ideal of $\mathcal{T}$. Let $\mathcal{C}_{\mathcal{T}'}$
denote the Serre span of $L(\tau), \tau\in \mathcal{T}'$. This is a highest weight category
with respect to the restriction of $\leqslant$ to $\mathcal{T}'$ with standard
objects $\Delta(\tau)$ and costandard objects $\nabla(\tau)$, where $\tau\in \mathcal{T}'$.
The quotient  $\mathcal{C}/\mathcal{C}_{\mathcal{T}'}$ is a highest weight category
with respect to the restriction of $\leqslant$ to $\mathcal{T}\setminus \mathcal{T}'$.
The natural functor $D^b(\Cat_{\mathcal{T}'})\rightarrow D^b(\Cat)$
is a full embedding. Further, the quotient $D^b(\Cat)/D^b(\Cat_{\mathcal{T}'})$
coincides with $D^b(\Cat/\Cat_{\mathcal{T}'})$.

Let us finish by recalling the Ringel duality. Set $T:=\bigoplus_{\tau\in \mathcal{T}}T(\tau)$.
Consider the category $\Cat^\vee:=\operatorname{End}_{\Cat}(T)^{opp}\operatorname{-mod}$.
Set $\Delta^\vee(\tau):=\operatorname{Hom}(T,\nabla(\tau))$. Then $\Cat^\vee$ is a highest weight
category with respect to the poset $\mathcal{T}^{opp}$. It is called the Ringel dual of
$\mathcal{C}$. The functor $\mathcal{R}:=R\operatorname{Hom}(T,\bullet)$ is
a derived equivalence $\mathcal{R}:D^b(\Cat)\xrightarrow{\sim} D^b(\Cat^\vee)$ to be called
the Ringel duality functor.
%We note that
%$(\Cat^{\vee})^{\vee}$ is naturally identified with $\Cat^{opp}$ and that
%$(\Cat^{opp})^\vee$ is naturally identified with $(\Cat^\vee)^{opp}$.
%So $\Cat$ is the Ringel dual of $(\Cat^\vee)^{opp}$. We

\subsection{Perverse equivalences}
\subsubsection{Definition and combinatorial data}
Let $\mathcal{T}^1,\mathcal{T}^2$ be  triangulated categories equipped with  $t$-structures. Let $\mathcal{C}^1,\mathcal{C}^2$ denote the hearts of $\mathcal{T}^1,\mathcal{T}^2$, respectively.
We are going to recall the notion of a perverse equivalence with respect to  filtrations $\mathcal{C}^i=\mathcal{C}_0^i\supset \mathcal{C}_1^i
\supset\ldots \supset\mathcal{C}_k^i=\{0\}$ by Serre subcategories, see \cite[Section 2.6]{Rouquier_ICM}
(there the definition is given for derived categories, but it generalizes to triangulated
categories in a straightforward way). By definition, this is a triangulated equivalence
$\mathcal{T}^1\rightarrow \mathcal{T}^2$ subject to the following
conditions:
\begin{itemize}
\item[(P1)] For any $j$, the equivalence $\mathcal{F}$ restricts to an equivalence
$\mathcal{T}^1_{\mathcal{C}_j^1}\rightarrow \mathcal{T}^2_{\mathcal{C}_j^2}$, where
we write $\mathcal{T}^i_{\mathcal{C}_j^i}, i=1,2,$ for the category of all objects
in $\mathcal{T}^i$ with homology in $\mathcal{C}_j^i$.
\item[(P2)] For $M\in \Cat_j^1$, we have $H_\ell(\mathcal{F}M)=0$ for $\ell<j$.
\item[(P3)] The  functor $M\mapsto H_j(\mathcal{F}M)$ induces an equivalence $\Cat^1_j/\Cat^1_{j+1}\xrightarrow{\sim}
\Cat^2_j/\Cat^2_{j+1}$  of abelian categories. Moreover,  for $M\in \Cat^1_j$
and  $\ell>j$, we have $H_\ell(\mathcal{F}M)\in \Cat^2_{j+1}$.
\end{itemize}

To $\mathcal{F}$ we assign its {\it combinatorial data}: the map
$\varphi=(\varphi_b, \varphi_s):\operatorname{Irr}\Cat^1\rightarrow \operatorname{Irr}\Cat^2\times \Z_{\geqslant 0}$.
To $M\in \operatorname{Irr}(\Cat^1_j/\Cat^1_{j+1})\subset \operatorname{Irr}(\Cat^1)$
it assigns the pair $(H_j(\mathcal{F}M), j)$, where the first component
is viewed as an element of $\operatorname{Irr}(\Cat^2_j/\Cat^2_{j+1})\subset
\operatorname{Irr}(\Cat^2)$. Note that $\varphi_b$ is a bijection.
The second component of the map is called  the {\it homological shift}.

Note, in particular, that any t-exact equivalence is perverse, where all homological shifts
are 0.

\subsubsection{Basic properties}
The following lemma is trivial.

\begin{Lem}\label{Lem:perv_compos_ab}
Let $\mathcal{F}:\mathcal{T}^1\rightarrow \mathcal{T}^2$ be a perverse equivalence
with perversity data $\varphi$. Further, let $\mathcal{G}$  be a $t$-exact self-equivalence of $\mathcal{T}^1$
inducing a bijection $\psi$ on $\operatorname{Irr}(\mathcal{C}^1)$ and $\mathcal{G}'$ be a t-exact
self-equivalence of $\mathcal{T}^2$ inducing a bijection $\psi'$ on $\operatorname{Irr}(\mathcal{C}^2)$.
Then the equivalences $\mathcal{F}\circ \mathcal{G}$ and $\mathcal{G}'\circ \mathcal{F}$
are perverse with perversity data $(\varphi_b\circ \psi,\varphi_s\circ\psi)$ and $(\psi'\circ \varphi_b,\varphi_s\circ \psi'^{-1})$.
\end{Lem}

The following important lemma is standard.

\begin{Lem}\label{Lem:perv}
Let $\mathcal{F},\mathcal{F}': \mathcal{T}^1\rightarrow \mathcal{T}^2$ be two perverse
equivalences with the same combinatorial data. Then there are  t-exact self-equivalences $\mathcal{G}'$
of  $\mathcal{T}^2$ inducing the identity
map on $\operatorname{Irr}(\mathcal{C}^2)$
and $\mathcal{G}$ of $\mathcal{T}^1$ inducing the identity map on
$\operatorname{Irr}(\mathcal{C}^1)$ with
$\mathcal{F}'=\mathcal{G}'\circ \mathcal{F}=\mathcal{F}\circ \mathcal{G}$.
\end{Lem}

\begin{Cor}\label{Cor:perv_abel_precomp}
Let $\mathcal{F}$ be a perverse equivalence $\mathcal{T}^1\rightarrow \mathcal{T}^2$
and let $\mathcal{G}$ be a t-exact self-equivalence of $\mathcal{T}^1$
giving the identity on $\operatorname{Irr}(\mathcal{C}^1)$. Then there are
t-exact self-equivalence $\mathcal{G}'$ of $\mathcal{T}^2$ giving the
identity on $\operatorname{Irr}(\mathcal{C}^2)$ such that $\mathcal{G}'\circ \mathcal{F}
=\mathcal{F}\circ \mathcal{G}$. A similar claim holds for $\mathcal{G},
\mathcal{G}'$ swapped.
\end{Cor}
\begin{proof}
Apply Lemma \ref{Lem:perv} to $\mathcal{F}$ and $\mathcal{F}\circ \mathcal{G}$.
\end{proof}

\begin{Lem}\label{Lem:K_0_perv}
Suppose that $\mathcal{C}^1,\mathcal{C}^2$ have finitely many irreducible objects. Let $\mathcal{F},
\mathcal{F}'$ be perverse equivalences $\mathcal{T}^1\rightarrow \mathcal{T}^2$.
If the induced maps $[\mathcal{F}],[\mathcal{F}']:
K_0(\mathcal{C}^1)\rightarrow K_0(\mathcal{C}^2)$ coincide, then so do
the bijections $\varphi_b,\varphi_b':\operatorname{Irr}(\mathcal{C}^1)\rightarrow \operatorname{Irr}(\mathcal{C}^2)$.
\end{Lem}
\begin{proof}
Let $\Cat^1=\Cat^1_0\supset \Cat^1_1\supset\ldots\supset \Cat^1_k=\{0\},
\Cat^2=\Cat^2_0\supset \Cat^2_1\supset\ldots\supset \Cat^2_k=\{0\}$ be the filtrations
for $\mathcal{F}$. We will prove by the descending induction on $i$ that, for
$L\in \operatorname{Irr}(\mathcal{C}^1_i)$, we have $\varphi_b(L)=\varphi'_b(L)$.
The case $i=k$ is vacuous. Now suppose that we know that $\varphi_b(L_1)=\varphi'_b(L_1)$
for any $L_1\in \operatorname{Irr}(\mathcal{C}^1_{i+1})$. Let $L\in \operatorname{Irr}(\mathcal{C}^1_i)
\setminus \operatorname{Irr}(\mathcal{C}^1_{i+1})$. Then
$$[\mathcal{F}L]= \pm [\varphi_b(L)]+\sum_{L_2\in \operatorname{Irr}(\mathcal{C}^2_{i+1})}a_{L_2}[L_2]$$
with $a_{L_2}\in \Z$. On the other hand, $[\varphi_b'(L)]$ appears with nonzero coefficient
in $[\mathcal{F}'L]$. Since $[\mathcal{F}L]=[\mathcal{F}'L]$, we see that either
$\varphi_b'(L)=\varphi_b(L)$ or $\varphi_b'(L)\in \operatorname{Irr}(\mathcal{C}^2_{i+1})$.
But by the inductive assumption, any object $L_2\in \operatorname{Irr}(\mathcal{C}^2_{i+1})$
is of the form $\varphi_b'(L_1)$ for $L_1\in \operatorname{Irr}(\Cat^1_{i+1})$.
We conclude that $\varphi_b(L)=\varphi_b(L')$ and this finishes the induction step.
\end{proof}

%The following lemma is straightforward.
%
%\begin{Lem}\label{Lem:perv_K0}
%Suppose that $\mathcal{C}^1,\mathcal{C}^2$ are equivalent to the categories of
%finite dimensional modules over finite dimensional algebras. Let $\mathcal{F},
%\mathcal{F}'$ be perverse equivalences $\mathcal{T}^1\rightarrow \mathcal{T}^2$
%with respect to the same filtrations. If the induced maps $[\mathcal{F}],[\mathcal{F}']:
%K_0(\mathcal{C}^1)\rightarrow K_0(\mathcal{C}^2)$ coincide, then so do the perversity data
%for $\mathcal{F},\mathcal{F}'$.
%\end{Lem}

\section{Preliminaries on the representation theory of $U(\g)$}\label{S_Ug_rep}
\subsection{Category $\mathcal{O}$}
\subsubsection{Definition of $\mathcal{O}$ and a highest weight structure}\label{SSS_O}
Let $\g$ be a semisimple Lie algebra over $\C$. Let $\h\subset \g$ be a Cartan subalgebra and let $\bo\subset \g$
be a Borel subalgebra containing $\h$. We write $W$ for the Weyl group of $\g$.

Set $\U=U(\g)$. We identify the center of $\U$ with $S(\h)^W$ by the Harish-Chandra isomorphism.
So, for $\lambda\in \h^*$, we can consider the central reduction $\U_\lambda$. Obviously,
$\U_{w\lambda}$ is naturally isomorphic to $\U_\lambda$ for any $\lambda\in \h^*$
and any $w\in W$.

We consider the category $\mathcal{O}_\lambda$ of all finitely generated $\U_\lambda$-modules
with locally finite $\bo$-action.  When $\lambda$ is regular
(meaning that $\langle \lambda,\alpha^\vee\rangle\neq  0$
for all roots $\alpha$), the category $\mathcal{O}_\lambda$
is highest weight, its standard objects are Verma modules
$\Delta(w\lambda), w\in W$. The order is as follows:
$w\lambda\preceq w'\lambda$ if there is a sequence of
(non-necessarily simple) reflections $s_1,\ldots,s_k\in W$
such that $w'=s_k\ldots s_1 w$ and the difference
$s_is_{i-1}\ldots s_1 w\lambda-s_{i-1}\ldots s_1 w\lambda$
is a positive multiple of a positive root for any $i$.

When $\lambda$ is regular, integral and anti-dominant, we identify the poset $W\lambda$
with $W$ by sending $w\in W$ to $w\lambda$. We write $\Delta_w$ for
$\Delta(w\lambda)$ and $L_w$ for  $L(w\lambda)$. We will write
$\mathcal{O}(W)$ for the corresponding category $\mathcal{O}_\lambda$
(that is independent of the choice of $\lambda$ up to a translation
equivalence). Note that a highest weight order for $\OCat(W)$
is the Bruhat order on $W$.

To finish this section, we recall that the natural functor
$D^b(\mathcal{O}_\lambda)\rightarrow D^b(\U_\lambda\operatorname{-mod})$
is a fully faithful embedding provided $\lambda$ is regular.

\subsubsection{Equivalences between categories $\mathcal{O}$ and HC bimodules}
Let us start by introducing various versions of the category of
Harish-Chandra (shortly, HC) bimodules.

By a HC $\U$-bimodule, we mean a finitely generated $\U$-bimodule with locally finite
adjoint action of $\g$.  Pick $\lambda,\mu\in \h^*$. We write
$\HC_{\lambda,\mu}^{1,\infty}(\U)$ for the category of all HC $\U$-bimodules
with central character $\lambda$ on the left and generalized central character
$\mu$ on the right. The notations $\HC_{\lambda,\mu}^{1,1}(\U)$ or $\HC_{\mu,\lambda}^{\infty,1}(\U)$
have  similar meanings.  Note that the categories $\HC_{\lambda,\mu}^{1,\infty}(\U),
\HC_{\mu,\lambda}^{\infty,1}(\U)$ are naturally equivalent (by switching the left and right actions
of $\U$ and twisting them by the antipode map for $\U$). We denote this equivalence
by $X\mapsto X^{op}$.

Now let $\mathcal{O}_\lambda'$ denote the infinitesimal block of the BGG category $\mathcal{O}$
with generalized central character $\lambda$.
The modules there are finitely generated over $\U$, have
a locally finite $\bo$-action and diagonalizable action of Cartan. Again, for $\lambda$
regular and integral, we write $\OCat'(W)$ for $\OCat'_\lambda$.
For $\lambda$ regular, integral and dominant, a classical result of Bernstein and Gelfand, \cite{BG},
establishes an equivalence $\HC^{\infty,1}_{\mu,\lambda}(\U)\xrightarrow{\sim}\OCat'_\mu$
that sends a HC bimodule $X$ to $X\otimes_{\U}\Delta(\lambda)$.

There is also an equivalence between $\OCat_\lambda$ and $\HC_{\lambda,\mu}^{1,\infty}$
($\mu$ is regular, integral and dominant) due to Soergel, \cite{Soergel}.
It sends $X\in \HC_{\lambda,\mu}^{1,\infty}$ to $\varprojlim_{k\rightarrow\infty}
X\otimes_{\U}\Delta^k(\mu)$, where $\Delta^k(\mu)$ stands for $\U\otimes_{\U(\bo)}U(\h)/\mathfrak{m}_\lambda^k$
with $\mathfrak{m}_\lambda$ denoting the maximal ideal of $\mu$ in $U(\h)$.

Composing the equivalences
$$\OCat_\lambda\xrightarrow{\sim}\HC^{1,\infty}_{\lambda,\lambda}\xrightarrow{\sim}
\HC^{\infty,1}_{\lambda,\lambda}\xrightarrow{\sim}\OCat'_\lambda$$
we get an equivalence $\OCat_\lambda\xrightarrow{\sim}\OCat'_\lambda$
that sends $\Delta(w\lambda)$ to $\Delta(w^{-1}\lambda)$. Here $\lambda$
is regular, integral and antidominant.

We will need a corollary of this equivalence. Namely, pick a subdiagram $D_1\subset D$
and let $W_1$ denote the corresponding parabolic subgroup. Pick a right $W_1$-coset,
say $c$. This is an interval in the Bruhat order. Consider the highest weight subcategories
$\OCat_{\prec c}(W)\subset \OCat_{\preceq c}(W)\subset \OCat(W)$. Here we write
$\OCat_{\preceq c}(W)$ for the subcategory corresponding to the poset ideal
$\{w\in W| w\preceq w'\text{ for some }w'\in c\}$, the subcategory $\OCat_{\prec c}(W)$
is defined similarly. Form the quotient $\OCat_c(W):=\OCat_{\preceq c}(W)/\OCat_{\prec c}(W)$.

\begin{Lem}\label{Lem:quot_parab}
There is an equivalence $\OCat_{c}(W)\cong \OCat(W_1)$ which sends $\Delta_{w^1 w}$ to
$\Delta_w$, where $w^1$ is the shortest element in $c$.
\end{Lem}
\begin{proof}
Under the equivalence $\OCat(W)\cong \OCat'_\lambda$, the interval $c$ corresponds
to an  interval $c'$ in $W\lambda$ consisting of all $\mu$ with fixed pairing with $\mathfrak{z}(\mathfrak{l})$,
where $\mathfrak{l}$ denotes the standard Levi subalgebra in $\g$ corresponding to $D_1$.
Let $\mu$ denote an element in $c'$.
Let $\mathfrak{p}$ be the standard parabolic with Levi subalgebra $\mathfrak{l}$.
Consider the parabolic induction functor $\Delta_{\mathfrak{p}}:=\U\otimes_{U(\mathfrak{p})}\bullet:
\OCat_{\mu}'(W_1)\rightarrow \OCat'_{\lambda}$. Its image lies in
$\OCat'_{\lambda,\preceq c'}$ and the composition $\OCat_{\mu^-}'(W_1)\rightarrow
\OCat'_{\lambda,\preceq c'}/\OCat'_{\lambda,\prec c'}$ is easily seen to be an equivalence
(a quasi-inverse functor is given by taking an appropriate eigenspace for $\mathfrak{z}(\mathfrak{l})$).

This shows an equivalence  $\OCat_{c}(W)\cong \OCat(W_1)$. The claim about the images of Vermas
follows from the corresponding statement for Soergel's equivalences $\OCat(W)\cong \OCat'(W),
\OCat(W_1)\cong \OCat'(W_1)$.
\end{proof}

\subsubsection{Relation to Kazhdan-Lusztig bases and cells}
We identify  $K_0(\OCat(W))$ with $\Z W$ by sending $[\Delta_w]$
to $w$. Recall that the Kazhdan-Lusztig conjecture (proved by Beilinson-Bernstein
and Brylinski-Kashiwara) implies that $[L_w]$ is the specialization of
$C_w$ to $v=1$.

As a consequence, one has the following classical connection between the cells in
$W$ and simple modules in $\OCat(W)$. Let us write $\J_w$ for the annihilator
of $L(-w\rho)\in \OCat_\rho$. Then the following is true:
\begin{itemize}
\item We have $w\sim_L w'$ if $\J_w=\J_{w'}$.
\item We have $w\sim_{LR}w'$ if the associated varieties of $\J_w,\J_{w'}$
coincide.
\item We have $w\sim_R w'$ if $w\sim_{LR}w'$ and there is $X\in \HC^{1,1}_{\rho,\rho}(\U)$
such that $L_w$ is a composition factor of $X\otimes_\U L_{w'}$.
\end{itemize}

\subsection{Wall-crossing functors}
\subsubsection{Localization theorems}
Recall the Beilinson-Bernstein (abelian and  derived) localization theorems. Let $\lambda$ be a regular element in
$\h^*$. Recall that $\lambda$ is called dominant if $\langle \lambda,\alpha^\vee\rangle\not \in \Z_{\leqslant 0}$.
Let $G$ be a semisimple algebraic group with Lie algebra $\g$ and $B\subset G$ be the Borel
subgroup corresponding to $\bo$.

Form the sheaf $D^\lambda_{G/B}$ of $\lambda$-twisted differential operators on $G/B$.
We have the global section functor $\Gamma_\lambda: \operatorname{Coh}(D^\lambda_{G/B})\rightarrow
\U_\lambda\operatorname{-mod}$ and its left adjoint, the localization functor
$\Loc_\lambda:=D^{\lambda}_{G/B}\otimes_{\U_\lambda}\bullet$ (recall that $\Gamma(D^\lambda_{G/B})=\U_\lambda$).
Further, we have the derived functors $R\Gamma_\lambda: D^b(\operatorname{Coh}(D^\lambda_{G/B}))
\rightarrow D^b(\U_\lambda\operatorname{-mod})$ and $L\Loc_\lambda:D^-(\U_\lambda\operatorname{-mod})
\rightarrow D^-(\operatorname{Coh}(D^\lambda_{G/B}))$.

The following result is due to Beilinson and Bernstein.

\begin{Prop}\label{Prop:loc_thm}
The following is true.
\begin{enumerate}
\item The functors $\Gamma_\lambda$ is an equivalence if and only
if $\lambda$ is regular and dominant. Its quasi-inverse is $\Loc_\lambda$.
\item The functors $R\Gamma_\lambda$ is an equivalence if and only if
$\lambda$ is regular. Its quasi-inverse is $L\Loc_\lambda$.
\end{enumerate}
\end{Prop}

\subsubsection{Wall-crossing functors and braid group action}\label{SS_WC_fun}
Recall that $\OCat_\lambda=\OCat_{w\lambda}$ for any $\lambda$.
Assume that $\lambda$ is regular. For $w\in W$, let us write
$\lambda w$ for $w^{-1}\lambda$. Let $W^\lambda$ be the subgroup
of $W$ generated by reflections $s_\alpha$ such that $\langle \alpha^\vee,\lambda\rangle\in \Z$
(a.k.a. the integral Weyl group of $\lambda$).  Our goal here is to recall a self-equivalence
$\WC_w, w\in W^\lambda,$ of $D^b(\OCat_\lambda)$ (known as an intertwining functor, a twisting functor
or a wall-crossing functor, we use the latter name) and list its properties.
For details of the proofs, the reader may consult \cite[Section L.3]{Milicic}
or \cite[Section 2]{BMR} (that treats the positive characteristic case).

Assume that $w$ and $\lambda$ are such that $\lambda<\lambda w$ in the order
recalled in \ref{SSS_O}. Consider the equivalences
$L\Loc_{\lambda}: D^b(\U_\lambda\operatorname{-mod})\rightarrow D^b(\Coh(D^\lambda_{G/B})),
L\Loc_{\lambda w}: D^b(\U_\lambda\operatorname{-mod})\rightarrow D^b(\Coh(D^{\lambda w}_{G/B}))$.
We also have an abelian equivalence $\mathcal{T}_{\lambda w\leftarrow w}:
\Coh(D^\lambda_{G/B})\xrightarrow{\sim} \Coh(D^{\lambda w}_{G/B})$ given by
tensoring with the line bundle $\mathcal{O}(\lambda w-\lambda)$.

\begin{Prop}
The following is true.
\begin{enumerate}
\item A derived self-equivalence $R\Gamma_{\lambda w}\circ \mathcal{T}_{\lambda w\leftarrow \lambda}
\circ L\Loc_\lambda$ of $D^b(\U_\lambda\operatorname{-mod})$
depends only on $w$ and $W^{\lambda}\lambda$, not on $\lambda$ itself.
Denote it by $\WC_w$.
\item For simple roots $\alpha_1,\ldots,\alpha_k$ for $W^\lambda$, the map
$T_{\alpha_i}\mapsto \WC_{s_i}$ (where $s_i$ here denotes the simple reflection
in $W^\lambda$ corresponding to $\alpha_i$) gives a weak categorical action
of $\mathsf{Br}_{W^\lambda}$ on $D^b(\OCat_\lambda)$. The element $T_w\in \mathsf{Br}_{W^\lambda}$
maps to $\WC_{w}$.
\end{enumerate}
\end{Prop}

\subsubsection{Behavior on $K_0$}
Here we will compute the image of $\WC_w$ on $K_0(\mathcal{O}_\lambda)$.
Recall that $K_0(\OCat_\lambda)$ is identified with $\Z W$.

\begin{Lem}\label{Lem:WC_K_0}
The map $[\WC_w]: \Z W\mapsto \Z W$ is given by $u\mapsto u w^{-1}$.
\end{Lem}
\begin{proof}
It is enough to check this on the generators $s_i$. In this case, the functor $\WC_{s_i}$
is known to coincide with the classical wall-crossing functor with respect to the wall
$\alpha_i=0$ and our claim follows. Recall that $\WC_{s_i}$ is defined as follows.
Let $\mathcal{T}_i$ be the translation functor to the wall $\alpha_i=0$. Then
$\WC_{s_i}(M)$ is quasi-isomorphic to the complex $M\rightarrow \mathcal{T}\mathcal{T}^*(M)$.
\end{proof}

\subsubsection{Wall-crossing bimodules}\label{SS_WC_bim}
In fact, the functor $\WC_{w}$ can be realized as the derived tensor
product with a Harish-Chandra bimodule. Basically, all constructions
of this part can be found in \cite[Sections 6.3,6.4]{BPW} in a more general
setting.

Namely, suppose that $\lambda w$
is regular and dominant so that, automatically, $\lambda<\lambda w$. Lift $\mathcal{O}_{\lambda w-\lambda}$
to a line bundle on $T^*(G/B)$. This line bundle admits a unique deformation to a
$D^{\lambda w}_{G/B}$-$D^{\lambda}_{G/B}$-bimodule to be denoted by $\B^{loc}_{\lambda w\leftarrow \lambda}$.
We have $\mathcal{T}_{\lambda w\leftarrow \lambda}=\B^{loc}_{\lambda w\leftarrow \lambda}\otimes_{D^{\lambda}_{G/B}}\bullet$. We set $\B_{\lambda w\leftarrow \lambda}:=\Gamma(\B^{loc}_{\lambda w\leftarrow \lambda})$. So we get
\begin{equation}\label{eq:wc_bimod}\WC_{w}=\B_{\lambda w\leftarrow \lambda}\otimes^L_{\U_{\lambda}}\bullet.
\end{equation}

Below we will need a deformation of $\B_{\lambda w\leftarrow \lambda}$. Assume that $\lambda$ is integral
and $\lambda w$ is dominant. Let $\param_0\subset \h^*$ be a subspace fixed by $w$. Set $\param_1:=\lambda+\param_0,
\chi:=\lambda w-\lambda$. Consider the deformation $X_{\param_0}:=G\times_B (\h_0\oplus \mathfrak{n})$
of $T^*(G/B)=G\times_B \mathfrak{n}$, where $\mathfrak{n}$ is the nilpotent radical of $\mathfrak{b}$
and $\h_0\subset \h$ corresponds to $\param_0$ under an identification $\h\cong \h^*$
coming from the Killing form.
We still have the line bundle $\mathcal{O}_{\chi}$ on  $X_{\param_0}$, the lift of $\mathcal{O}_\chi$
from $G/B$. Consider the quantizations
$\mathcal{D}^{\param_1}_{G/B}, \mathcal{D}^{\param_1 w}_{G/B}$ of $X_{\param_0}$. The bundle
$\mathcal{O}_{\chi}$ quantizes into the $\mathcal{D}^{\param_1 w}_{G/B}$-$\mathcal{D}^{\param_1}_{G/B}$-bimodule
$\B^{loc}_{\param_1,\chi}$. We set $\B_{\param_1,\chi}:=\Gamma(\B^{loc}_{\param_1,\chi})$.

A relation between the bimodules $\B_{\lambda' w\leftarrow \lambda'}$ for $\lambda'\in \param_1$
and $\B_{\param_1,\chi}$ is given by the following lemma that follows from \cite[Proposition 6.25]{BPW}
combined with the Beilinson-Bernstein localization theorem.

\begin{Lem}\label{Lem:WC_spec}
Suppose $\lambda' w$ is regular and dominant. Then $\B_{\lambda' w\leftarrow \lambda'}$
coincides with the specialization $\B_{\lambda',\chi}:=\B_{\param_1,\chi}\otimes_{\C[\param_1]}\C_{\lambda'}$.
\end{Lem}

\subsubsection{Long wall-crossing vs Ringel duality}
Set $n:=\dim \mathfrak{n}=\dim G/B$.
%The following result was essentially established in \cite[Proposition 4.7]{BFO}
%(a more general result follows from \cite[Section 4]{BL} that easily extends to
%the present setting).

\begin{Lem}\label{Long_wc_perv}
Let $\lambda$ be regular and $w_0\in W^\lambda$ denote the longest element.
Then the functor $\WC_{w_0}:D^b(\OCat_\lambda)\rightarrow D^b(\OCat_\lambda)$
is perverse, where both filtrations are by codimension of support:
the subcategory $\Cat^i_j\subset \OCat_\lambda, i=1,2,$ consists of all modules
$M$ such that $n-\operatorname{GK-}\dim M\geqslant j$.
Moreover, there is an abelian equivalence $\OCat_\lambda^\vee\cong \OCat_\lambda$
that intertwines the inverse Ringel duality functor $\mathcal{R}^{-1}:\OCat_\lambda^\vee\rightarrow 
\OCat_\lambda$  with $\WC_{w_0}$.
\end{Lem}
\begin{proof}
The functor $\WC_{w_0}$ coincides with the homological duality functor $R\operatorname{Hom}_{\U_\lambda}(\bullet,\U_\lambda)$ up to precomposing
with an abelian equivalence. This is proved in \cite[Section 4.3]{BFO}
in a slightly different setting (see also \cite[Section 4]{BL}, the proof
used there can be adapted to our setting verbatim). That the homological
dulaity functor is perverse is a bit implicit in \cite[Section 4.3]{BFO}
and is more explicit in \cite[Lemma 2.5]{rouq_der}. That the homological duality
functor coincides with the inverse Ringel duality (up to precomposing with
an abelian equivalence) is checked in \cite[Section 4.1]{Gies} (again,
the proof carries to the present case verbatim).
\end{proof}

\subsection{W-algebras}
\subsubsection{Construction and basic properties}
Pick a nilpotent orbit $\Orb\subset \g$ and an element $e\in \Orb$. Consider the Slodowy
slice $S\subset \g$ that is a transverse slice to $\Orb$ in $\g$. It is constructed
as follows: we pick an $\mathfrak{sl}_2$-triple $(e,h,f)$ and set
$S:=e+\z_{\g}(f)$. This is an affine space equipped with a so called Kazdhan action
of $\C^\times$ that contracts it to $e$,
the action is given by $t.s:=t^{-2}\gamma(t)s$, where $\gamma:\C^\times\rightarrow G$
corresponds to $h$. The algebra $\C[S]$ is graded Poisson with bracket of degree $-2$.

The algebra $\C[S]$ admits a distinguished quantization called the finite
W-algebra that was first constructed by Premet in \cite{Premet1}. Let us recall
a construction in the version of Gan and Ginzburg, \cite{GG}. 

The element $h$ induces a grading on
$\g$, $\g=\bigoplus_{i\in \Z}\g(i)$. Let $(\cdot,\cdot)$ denote the  Killing
form. We set $\chi:=(e,\cdot)$ and $\g(\leqslant j):=\bigoplus_{i\leqslant j}\g(i)$.
The space $\g(-1)$ is symplectic with  respect to the form $(x,y)
\mapsto \langle \chi, [x,y]\rangle$. Pick a lagrangian subspace $\ell\subset \g(-1)$
and set $\mathfrak{m}:=\ell\oplus \g(\leqslant -2)$. Let $M\subset G$
denote the corresponding connected subgroup, it is unipotent. The subgroup
$M$ acts on $\g^*\cong \g$ in a Hamiltonian way, the moment map is just the
restriction map $\g^*\rightarrow \mathfrak{m}^*$. It is easy to see that
$S\subset \mu^{-1}(\chi)$. It turns out that the action map $M\times S\rightarrow
\mu^{-1}(\chi)$ is an isomorphism, see \cite[Lemma 2.1]{GG}, so $\C[S]=[S(\g)/S(\g)\m_\chi]^M$, where
$\m_\chi:=\{x-\langle\chi,x\rangle| x\in \m\}$ is viewed as a subspace in $\g\oplus \C$.

This motivates the following definition: $\Walg:=[U(\g)/U(\g)\m_\chi]^M$. This is a filtered
associative algebra, the filtration is inherited from the so called Kazhdan filtration
on $U(\g)$, where $\deg \g(i)=i+2$. We have $\gr\Walg=\C[S]$ (where the grading on
$\C[S]$ is introduced in a similar fashion), see \cite[Proposition 5.2]{GG}.
It was also shown in \cite[Section 5.5]{GG} that a natural homomorphism
$[\U/\U\g(\leqslant -2)_{\chi}]^{G(\leqslant -1)}\subset \Walg$ is an isomorphism.
This gives rise to a Hamiltonian action of $Q:=Z_G(e,h,f)$
on $\Walg$. The center of $U(\g)$ naturally maps into $\Walg$ and this map
is an isomorphism onto the center of $\Walg$, see the footnote
in \cite[Section 5.7]{Premet2}. Because of this, for $\lambda\in \h^*$,
we have a central reduction $\Walg_\lambda$ of $\Walg$, quantizing
$\C[S\cap \mathcal{N}]$, where $\mathcal{N}$ stands for the nilpotent cone in  $\g$.

Finally, let us mention Skryabin's equivalence between the category of $\Walg$-modules
and a suitable full subcategory of {\it Whittaker} modules (=modules, where $\m$
acts with generalized eigen-character $\chi$) in the category of $U(\g)$-modules. Namely,
we have a functor $M\mapsto M^{\m_\chi}:U(\g)\operatorname{-mod}\rightarrow
\Walg\operatorname{-mod}$ that has the left adjoint $\mathsf{Sk}:N\mapsto [U(\g)/U(\g)\m_\chi]\otimes_{\Walg}N$.
Skryabin in the appendix to \cite{Premet1} has proved that these two functors
are mutually inverse equivalences.

%This space has a natural filtered quantization, the W-algebra $\Walg$
\subsubsection{Restriction functor for HC bimodules}\label{SSS_HC_bim_restr}
There is a notion of a $Q$-equivariant HC
$\Walg$-bimodule, \cite[Section 2.5]{HC}.  Let $\HC(\Walg,Q)$
denote the category of such bimodules. The notation
$\HC^{?,!}_{\lambda,\mu}(\Walg,Q)$ has the same meaning as for the universal
enveloping algebras.

According to \cite[Section 3.4]{HC}, there is an exact tensor {\it restriction} functor
$\HC(\U)\rightarrow \HC(\Walg,Q)$ to be denoted
by $\bullet_{\dagger}$. This functor is compatible with the associated varieties:
$\VA(X_{\dagger})=\VA(X)\cap S$. In particular, if $\VA(X)=\overline{\Orb}$,
then $X_{\dagger}$ is finite dimensional. Moreover, the restriction of
$\bullet_{\dagger}$ to the full subcategory $\HC(\U)_{\overline{\Orb}}$
of all HC bimodules $X$ with $\VA(X)=\overline{\Orb}$ is a quotient functor onto its
image. The kernel is the category $\HC(\U)_{\partial \Orb}$ of all HC
bimodules  $X$ with $\VA(X)\subset\partial\Orb$. The image is closed under
taking subquotients. In particular, if $X\in \HC(\U)_{\overline{\Orb}}$ is simple,
then $X_{\dagger}$ is a simple  $Q$-equivariant finite dimensional $\Walg$-bimodule.
Another useful property is that $\bullet_{\dagger}$ intertwines Tor's and Ext's,
compare with \cite[Lemma 3.11]{rouq_der}
and \cite[Section 5.5]{BL}. Let us also point out that by the construction given in
\cite[Section 3.4]{HC}, the functor $\bullet_{\dagger}$ maps
$\HC_{\lambda,\mu}^{?,!}(\U)$ to $\HC_{\lambda,\mu}^{?,!}(\Walg,Q)$,
where $?,!\in \{1,\infty\}$.

The functor $\bullet_{\dagger}: \HC(\U)_{\overline{\Orb}}\rightarrow \HC(\Walg,Q)_{fin}$
admits a right adjoint $\bullet^{\dagger}$. The cokernel of $\B\mapsto (\B_{\dagger})^{\dagger}$
is in $\HC(\U)_{\partial\Orb}$ for any $\B\in \HC(\U)_{\overline{\Orb}}$. The functor
$\bullet^{\dagger}$ maps $\HC_{\lambda,\mu}^{1,1}(\Walg,Q)_{fin}$ to
$\HC_{\lambda,\mu}^{1,1}(\U)_{\overline{\Orb}}$ for any $\lambda,\mu$.
For a $Q$-stable ideal $\I\subset \Walg$ of finite codimension, we write
$\I^{\dagger}$ for the kernel of the natural map $\U\rightarrow (\Walg/\I)^{\dagger}$.

These constructions admit a ramification that will be used later.
Namely, let $\param_1\subset \h^*$ be an affine subspace and let $\chi\in\h^*$.
Consider the category $\HC^{1,1}_{\param_1,\chi}(\U)$ of all HC $\U_{\param_1+\chi}$-$\U_{\param_1}$-bimodules,
where the adjoint action of $\h$ (that maps naturally to both $\C[\param_1],\C[\param_1+\chi]$)
is given by $\chi$. Here we set $\U_{\param_1}:=\C[\param_1]\otimes_{\C[\h^*]^W}\U$
and $\U_{\param_1+\chi}:=\C[\param_1+\chi]\otimes_{\C[\h^*]^W}\U$.  Define the category
$\HC^{1,1}_{\param_1,\chi}(\Walg,Q)$ in a similar fashion. We consider the subcategory
$\HC^{1,1}_{\param_1,\chi}(\U)_{\overline{\Orb}}\subset \HC^{1,1}_{\param_1,\chi}(\U)$
consisting of all bimodules $\B$ such that the intersection of $\VA(\B)\subset
\param_0\times_{\h^*/W}\g^*$ with $\mathcal{N}$ lies in $\overline{\Orb}$. Here we write
$\param_0\subset \h^*$ for the associated vector subspace of $\param_1$. Similarly,
define the subcategory $\HC^{1,1}_{\param_1,\chi}(\Walg,Q)_{fin}$, it consists of
all bimodules in $\HC^{1,1}_{\param_1,\chi}(\Walg,Q)$ that are finitely generated over
$\C[\param_1]$. We then have an exact functor
$\HC^{1,1}_{\param_1,\chi}(\U)\rightarrow \HC^{1,1}_{\param_1,\chi}(\Walg,Q)$
that restricts to $\HC^{1,1}_{\param_1,\chi}(\U)_{\overline{\Orb}}\rightarrow
\HC^{1,1}_{\param_1,\chi}(\Walg,Q)_{fin}$. The restriction  has a right adjoint
$\bullet^{\dagger}:\HC^{1,1}_{\param_1,\chi}(\Walg,Q)_{fin}\rightarrow
\HC^{1,1}_{\param_1,\chi}(\U)_{\overline{\Orb}}$. Note that, for $\chi,\chi'\in \h^*$,
we have tensor product functors \begin{align*}&\HC^{1,1}_{\param_1+\chi,\chi'}(\U)\times\HC^{1,1}_{\param_1,\chi}(\U)
\rightarrow \HC^{1,1}_{\param_1,\chi+\chi'}(\U),\\
&\HC^{1,1}_{\param_1+\chi,\chi'}(\Walg,Q)\times\HC^{1,1}_{\param_1,\chi}(\Walg,Q)
\rightarrow \HC^{1,1}_{\param_1,\chi+\chi'}(\Walg,Q).\end{align*}
The functor $\bullet_{\dagger}$ intertwines those functors.

The functor $\bullet_{\dagger}$ also can be defined using the quantum Hamiltonian
reduction, it maps a HC $\U$-bimodule $X$ to $[X/X\m_\chi]^M$, \cite[Section 3.5]{HC}.
Note that $\mathsf{Sk}^{-1}(X\otimes_{U(\g)}\mathsf{Sk}(N))$
is naturally identified with $X_{\dagger}\otimes_{\Walg}N$. This is a special case of
\cite[Theorem 5.11]{LO}.

\subsubsection{Classification of finite dimensional irreducible representations}
We are going to recall a classification of the finite dimensional irreducible
modules over $\Walg_\lambda$, \cite{HC,LO}. Let us start with results obtained in \cite{HC}.

Let $\operatorname{Pr}(\U_\lambda)$ denote the set of primitive ideals in $\U_\lambda$.
Inside we have the subset $\operatorname{Pr}_{\Orb}(\U_\lambda)$ of all primitive
ideals $\J$ such that  $\VA(\U_\lambda/\J)=\overline{\Orb}$.
For $\J\in \operatorname{Pr}_{\Orb}(\U_\lambda)$, the ideal $\J_{\dagger}\subset \Walg$
is $Q$-stable, has finite codimension, and is maximal with these properties.
The maximal ideals containing $\J_{\dagger}$ are $Q$-conjugate. This allows to assign
an $A$-orbit in $\operatorname{Irr}_{fin}(\Walg_\lambda)$ to $\J$, in fact, every finite
dimensional irreducible $\Walg_\lambda$-module lies in one of these orbits. This
gives rise to a bijection $\operatorname{Irr}_{fin}(\Walg_\lambda)/A\xrightarrow{\sim}
\operatorname{Pr}_{\Orb}(\U_\lambda)$.

When $\lambda$ is integral, one can compute the $A$-orbit over a given primitive ideal,
this was done in \cite{LO}. Let us state the result in the case
when $\lambda$, in addition, is regular. Namely, let $\operatorname{Spr}_{\Orb}$ denote the
Springer $W\times A$-module corresponding to $\Orb$. The set
$\operatorname{Pr}_{\Orb}(\U_\lambda)$ is non-empty if and only if $\Orb$
is special and we will be assuming this from now on. So to $\Orb$ we can assign
the two-sided cell $\bc$. Then the Lusztig group $\bA$  can be defined
as the quotient of $A$ by the kernel of the $A$-action on $\operatorname{Hom}_W([\bc],
\operatorname{Spr}_{\Orb})$, where we write $[\bc]$ for the two-sided cell bimodule
viewed as a left $W$-module, see \cite[Section 13.1]{Lusztig_orange}. Recall that $\operatorname{Pr}_{\Orb}(\U_\lambda)$
is in a bijection with the set of the left cells inside $\bc$. It was shown
in \cite[Sections 6.4-6.8]{LO}  that for a left cell $\sigma\subset \bc$, there is a unique
(up to conjugacy) subgroup $H_\sigma\subset \bA$ such that $\C(\bA/H_\sigma)
\cong \Hom_W([\sigma], \operatorname{Spr}_{\Orb})$ (an isomorphism of
$A$-modules). The main result of \cite{LO}, Theorem 1.1 there, is that the $A$-orbit corresponding
to the primitive ideal indexed by $\sigma$ coincides with $\bA/H_\sigma$.

For brevity, let us write $Y_\lambda$ instead of $\operatorname{Irr}_{fin}(\Walg_\lambda)$.
Consider also the category $\mathfrak{J}^{\Orb}_\lambda$ of all semisimple
objects in $\operatorname{HC}^{1,1}_{\lambda,\lambda}(\U)_{\overline{\Orb}}/
\operatorname{HC}^{1,1}_{\lambda,\lambda}(\U)_{\partial\Orb}$. This category is closed under
taking the tensor products and internal Hom's, see \cite[Section 5.3]{LO}.
The tensor category $\mathfrak{J}^{\Orb}_\lambda$ acts on the category
$\Walg_{\lambda}\operatorname{-mod}_{fin}^{ss}$ that can be thought as the
category of sheaves of finite dimensional vector spaces on $Y_\lambda$.
The action is given by $X.V:=X_{\dagger}\otimes_{\Walg_\lambda}V$.
The category $\mathfrak{J}^{\Orb}_\lambda$ is identified with $\operatorname{Sh}^{\bA}(Y_\lambda\times Y_\lambda)$
(where the tensor product is given by convolving the sheaves) in such a way that
the action on $\Walg_{\lambda}\operatorname{-mod}^{ss}_{fin}$ becomes the
convolution action of $\operatorname{Sh}^{\bA}(Y_\lambda\times Y_\lambda)$ on
$\operatorname{Sh}(Y_\lambda)$, see \cite[Remark 7.7]{LO}.

\subsubsection{Localization theorems}
Here we are going to recall the localization theorem for W-algebras,
\cite{Ginzburg_HC,DK}. Modules for $\Walg_\lambda$  localize to
modules over a certain sheaf of non-commutative algebras on the {\it Slodowy variety}
$\widetilde{S}$. By definition, $\widetilde{S}$ is the preimage of the $S$ under the Springer morphism
$T^*(G/B)\rightarrow \g^*$. This is a smooth symplectic variety that coincides
with $\mu^{-1}_{T^*(G/B)}(\chi)/M$, where we write $\mu_{T^*(G/B)}$
for the moment map for the $M$-action on $T^*(G/B)$.
Note that the natural morphism $\tilde{S}\rightarrow S\cap \mathcal{N}$
is a resolution of singularities.

The sheaf of non-commutative algebras of interest will be denoted by $\Wf_\lambda$.
It is a sheaf in the conical topology (the topology where ``open''
means Zariski open and $\C^\times$-stable, where we consider the
Kazhdan action on $\widetilde{S}$). This sheaf is filtered in
such a way that the filtration is compete and separated and the associated
graded is $\mathcal{O}_{\widetilde{S}}$ so that $\Wf_\lambda$
is a filtered quantization of $\widetilde{S}$. By definition,
$\Wf_{\lambda}$ is the reduction $[\mathcal{D}^{\lambda}_{G/B}/\mathcal{D}^\lambda_{G/B}
\m_\chi]^M$. Since the action of $M$ on $\mu^{-1}_{T^*(G/B)}(\chi)$ is free, we see that
$\mathfrak{W}_\lambda$ is indeed a quantization of
$\tilde{S}$. Note that there is a natural filtered algebra homomorphism
$\Walg_\lambda\rightarrow \Gamma(\mathfrak{W}_\lambda)$. The associated graded of this homomorphism
is an isomorphism $\C[S\cap \mathcal{N}]\xrightarrow{\sim}\C[\tilde{S}]$ and so
we see that $\Gamma(\mathfrak{W}_\lambda)=\Walg_\lambda$. We also remark that the isomorphism
$\gr\mathfrak{W}_\lambda=\mathcal{O}_{\tilde{S}}$ together with $H^i(\tilde{S},\mathcal{O}_{\tilde{S}})=0$
for $i>0$ imply $H^i(\tilde{S},\mathfrak{W}_\lambda)=0$ for $i>0$.
%From here one deduces, using results of \cite{GG}, that $\Gamma(\Wf_\lambda)=\Walg_\lambda$ and the
%higher cohomology of $\Wf_\lambda$ vanish.

Let us proceed to  the localization theorems.
First of all, the category of coherent $\Wf_\lambda$-modules is equivalent
to the category $\operatorname{Coh}^{\m,\chi}(D^\lambda_{G/B})$
of twisted $(\m,\chi)$-equivariant coherent $D^\lambda_{G/B}$-modules
(the equivalence is again given by taking $\m$-semiinvariants with character $\chi$).
This is an easy consequence of the observation that the $M$-action on
$\mu^{-1}_{T^*(G/B)}$ is free.
Moreover, for $M\in \operatorname{Coh}(\mathfrak{W}_\lambda)$, we have
\begin{equation}\label{eq:inv_glob}
\Gamma(M)^{\m_\chi}=\Gamma(M^{\m_\chi}).
\end{equation}
The following claim is a direct corollary of (\ref{eq:inv_glob}) and (1) of
Proposition \ref{Prop:loc_thm}.

\begin{Prop}\label{Prop:W_loc}
Suppose $\lambda$ is regular and dominant. Then the global section functor
$\Gamma_\lambda: \operatorname{Coh}(\Wf_\lambda)\rightarrow \Walg_\lambda\operatorname{-mod}$
is a category equivalence. A quasi-inverse equivalence is given by the localization
functor $\Wf_\lambda\otimes_{\Walg_\lambda}\bullet$.
\end{Prop}

Let us proceed to the derived localization. By \cite[Theorem 1.6]{Bern_Lunts}, the naive derived
category $D^b(\operatorname{Coh}^{\m,\chi}(D^\lambda_{G/B}))$ is naturally
isomorphic to the equivariant derived category $D^b_{\m,\chi}(\operatorname{Coh}(D^\lambda_{G/B}))$
because the $M$-action on $\mu^{-1}_{T^*(G/B)}(\chi)$ is free. Similarly, $D^b(\Walg_\lambda\operatorname{-mod}^{\m,\chi})=
D^b_{\m,\chi}(\Walg_\lambda\operatorname{-mod})$. We deduce that if the functor
$R\Gamma_\lambda: D^b(\operatorname{Coh}(D^\lambda_{G/B}))\rightarrow D^b(\Walg_\lambda\operatorname{-mod})$
is an equivalence, then the same is true for $R\Gamma_\lambda:
\operatorname{Coh}^{\m,\chi}(D^\lambda_{G/B})\rightarrow
D^b(\Walg_\lambda\operatorname{-mod}^{\m,\chi})$. Together with
(\ref{eq:inv_glob}) and equivalences $\operatorname{Coh}^{\m,\chi}(D^\lambda_{G/B})\cong
\operatorname{Coh}(\mathfrak{W}_\lambda),
\U_\lambda\operatorname{-mod}^{\m,\chi}\cong \Walg_\lambda\operatorname{-mod}$
this yields the following result.
%We can consider the equivariant
%derived category $D^b_{\m,\chi}(\Coh(D^\lambda_{G/B}))$. Since the action of
%$M$ on $\mu^{-1}_{T^*(G/B)}(\chi)$ is free, we see by \cite[Theorem 1.6]{Bern_Lunts}
%that the natural functor $D^b(\Coh^{\m,\chi}(D^\lambda_{G/B}))\rightarrow D^b_{\m,\chi}(\Coh(D^\lambda_{G/B}))$
%is a category equivalence. We have a functor $\mathfrak{F}:D^b_{\m,\chi}(\Coh(D^\lambda_{G/B}))\rightarrow
%D^b(\Coh(\Wf_\lambda))$ given by $R\operatorname{Hom}(D^\lambda_{G/B}/D^\lambda_{G/B}\m_\chi,\bullet)$.
%By the above equivalence of the two equivariant derived categories, this functor is an
%equivalence that comes from the functor of taking $(\m,\chi)$-semiinvarinats between
%the abelian categories. An analogous functor $\mathcal{F}:
%D^b_{\m,\chi}(\U_\lambda\operatorname{-mod})\rightarrow
%D^b(\Walg_\lambda\operatorname{-mod})$ is an equivalence induced
%by $\mathsf{Sk}^{-1}:\U_\lambda\operatorname{-mod}^{\m,\lambda}\xrightarrow{\sim}
%\Walg_\lambda\operatorname{-mod}$. Note that $R\Gamma_\lambda\circ \mathfrak{F}$
%is naturally isomorphic to $\mathcal{F}\circ R\Gamma_\lambda$. From here we deduce
%the following result.

\begin{Prop}\label{Prop:W_der_loc}
Suppose that $\lambda$ is regular. Then the functor $R\Gamma_\lambda:D^b(\Coh(\Wf_\lambda))
\rightarrow D^b(\Walg_\lambda\operatorname{-mod})$ is a category equivalence.
A quasi-inverse functor is $L\Loc_\lambda:=\Wf_\lambda\otimes^L_{\Walg_\lambda}\bullet$.
\end{Prop}

\subsubsection{Wall-crossing bimodules for W-algebras}
Let $\lambda$ be regular and $w\in W^\lambda$ be such that $\lambda w$ is regular and dominant.
So we have an equivalence $\WC_{\lambda w\leftarrow \lambda}: D^b(\Walg_\lambda\operatorname{-mod})
\xrightarrow{\sim} D^b(\Walg_\lambda\operatorname{-mod})$ defined similarly to the case of
$\U_\lambda$, see \ref{SS_WC_fun}. 
For similar reasons, it is given by the tensor product with $\B^{\Walg}_{\lambda w\leftarrow
\lambda}\in \HC^{1,1}_{\lambda,\lambda}(\Walg,Q)$ that is the global sections of a suitable
quantized line bundle $\B^{\Walg,loc}_{\lambda w\leftarrow
\lambda}$, compare to \ref{SS_WC_bim}.

\begin{Lem}\label{Lem:WC_Walg}
We have an isomorphism $\B^{\Walg}_{\lambda w\leftarrow \lambda}\cong (\B_{\lambda w\leftarrow \lambda})_{\dagger}$.
\end{Lem}
\begin{proof}
It is sufficient to establish an isomorphism $\B^{\Walg}_{\lambda w\leftarrow \lambda}\otimes_{\Walg_\lambda}
N\cong (\B_{\lambda w\leftarrow \lambda})_{\dagger}\otimes_{\Walg_\lambda}N$ for any
$N\in \Walg_\lambda\operatorname{-mod}$. Recall, \ref{SSS_HC_bim_restr}, that the right hand side is
$\mathsf{Sk}^{-1}(\B_{\lambda w\leftarrow \lambda}\otimes_{\U_\lambda}\mathsf{Sk}(N))$.
Using that the localization and global section functors commute with the Skryabin equivalence,
we reduce to checking the equality
$\B^{\Walg,loc}_{\lambda w\leftarrow \lambda}\cong [\B^{loc}_{\lambda w\leftarrow \lambda}/\B^{loc}_{\lambda w\leftarrow
\lambda}]^{M}$. This follows from the observation that both quantize the line bundle $\mathcal{O}_\chi$
on $\tilde{S}$ to a $\Wf_{\lambda w}$-$\Wf_{\lambda}$-bimodule and that such a quantization is unique,
see, e.g., \cite[Section 5.1]{BPW}.
\end{proof}

\section{Cactus group actions}\label{S_cacti}
\subsection{Perversity of $\WC_{D_1}$, statement}
We fix a subdiagram $D_1$ in the Dynkin diagram $D$ of $\g$. Let $\param_0$ denote the span of the fundamental
weights corresponding to the vertices in $D\setminus D_1$. Pick a regular  integral weight $\lambda$
such that $\lambda w_{D_1}$ is dominant and set $\param_1:=\lambda+\param_0$, this is an affine subspace
in $\h^*$ containing $\lambda$. Finally, set $\chi:=\lambda w_{D_1}-\lambda$.

Define the algebras $\U_{\param_1}:=\C[\param_1]\otimes_{\C[\h^*]^W}U(\g), \Walg_{\param_1}:=\C[\param_1]
\otimes_{\C[\h^*]^W}\Walg$. Recall that we have a $\U_{\param_1}$-bimodule $\U_{\param_1,\chi}$ defined as the global sections of the $D_{G/B}^{\param_1+\chi}$-$D_{G/B}^{\param_1}$-bimodule
quantizing the line bundle $\mathcal{O}_\chi$. Recall also (Lemma \ref{Lem:WC_spec})
that when  $\lambda_1+\chi$ is regular dominant,
the specialization $\U_{\lambda_1,\chi}$ coincides with the wall-crossing bimodule
$\B_{\lambda_1 w_{D_1}\leftarrow \lambda_1}$. Further, set $\Walg_{\param_1,\chi}:=(\U_{\param_1,\chi})_{\dagger}$.

We need to produce to produce chains of ideals in the algebras $\U_{\param_0}, \U_{\param_0'}$.

\begin{Prop}\label{Prop:ideal_chain}
There is a chain of ideals $\U_{\param_1}=\J_{0,\param_1}\supset \J_{1,\param_1}\supset
\ldots\supset \J_{n+1,\param_1}=\{0\}$, where $n:=\dim G/B$, with the following property:
for a Weil generic  point $\lambda_1\in \param_0$, the specialization $\J_{i,\lambda_1}$
is the minimal ideal with $\dim \VA(\U_{\lambda_1}/\J_{i,\lambda_1})<2i$.
\end{Prop}
\begin{proof}
Observe that, for any nilpotent orbit $\Orb$ and any $\lambda_1\in \h^*$, the algebra
$\Walg_{\lambda_1}$ has finite length as a bimodule over itself, see, e.g., \cite[Theorems 1.2,1.3]{B_ineq}.
In particular, $\Walg_{\lambda_1}$ has a minimal ideal of finite codimension.
Similarly to the proof of \cite[Lemma 5.1]{rouq_der}, one can show that
there is an ideal $\I_{\param_1}\subset \Walg_{\param_1}$ such that
$\Walg_{\param_1}/\I_{\param_1}$ is finitely generated over $\C[\param_1]$
and, for a Weil generic $\lambda_1\in \param_1$, the specialization $\I_{\lambda_1}$
is the minimal ideal of finite codimension in $\Walg_1$. Then we set
$\J_{i,\param_1}:=\left(\bigcap_{\Orb} \I_{\Orb,\param_1}^{\dagger,\Orb}\right)^i$,
where the intersection is taken over all orbits $\Orb$ with $\dim \mathcal{N}-\dim \Orb<2i$,
$\I_{\Orb,\param_1}$ means the ideal $\I_{\param_1}$ for the W-algebra corresponding
to $\Orb$, and $\bullet^{\dagger,\Orb}$ has a similar meaning. Similarly to
\cite[Lemma 5.2]{rouq_der}, the ideal $\J_{i,\param_1}$ has required properties.
\end{proof}

Note that, by the construction, $(\J_{i,\lambda_1})^2=\J_{i,\lambda_1}$ when $\lambda_1$ is Weil generic.
Therefore the equality is true when $\lambda_1$ is Zariski generic as well. It follows
that the full subcategory  $\Cat_i\subset\OCat_{\lambda_1}$ consisting of all modules annihilated
by $\J_{n+1-i,\lambda_1}$ is a Serre subcategory.

\begin{Thm}\label{Thm:perversity}
For a Zariski generic $\lambda_1\in \param_1$, the functor $\mathcal{F}(\lambda_1):=\U_{\lambda_1,\chi}\otimes^L_{\U_{\lambda_1}}\bullet$
is a perverse self-equivalence of $D^b(\OCat_{\lambda_1})$
with respect  to the filtrations defined by the ideals $\J_{i,\lambda_1}$.
%The self-bijection $\mathfrak{wc}_D$ induced by $\mathcal{F}(\lambda_1)$ of the set %$W=\operatorname{Irr}(\mathcal{O}_{\lambda_1})$ is independent
%of $\lambda_1$ and squares to the identity.
\end{Thm}

In particular, we can always choose a regular integral $\lambda_1\in \h^*$ such that
$\lambda_1 w_{D_1}$ is dominant and $\mathcal{F}(\lambda_1)$ is perverse
with respect to the filtration above.

\subsection{Perversity of $\WC_{D_1}$, proof}
The proof of Theorem \ref{Thm:perversity} basically repeats the proof of \cite[Theorem 6.1]{rouq_der}.
We provide details here for readers convenience.

Let us write $\B_{\lambda_1}$ for $\U_{\lambda_1,\chi}$.

\begin{Lem}\label{Lem:perv_bimod}
For a Zariski generic $\lambda_1\in \param_1$, the following holds:
\begin{itemize}
\item[(a)] For all $i,j$, we have $\J_{j,\lambda_1}\operatorname{Tor}^{\U_{\lambda_1}}_i(\B_{\lambda_1}, \U_{\lambda_1}/\J_{j,\lambda_1})=0$.
\item[(b)] For all $i,j$, we have $\operatorname{Tor}^{\U_{\lambda_1}}_i(\U_{\lambda_1}/\J_{j,\lambda_1}, \B_{\lambda_1})\J_{j,\lambda_1}=0$.
\item[(c)] We have   $\operatorname{Tor}^{\U_{\lambda_1}}_i(\B_{\lambda_1}, \U_{\lambda_1}/\J_{j,\lambda_1})=0$
  for $i<n+1-j$.
%\item[(c)] We have $\operatorname{Tor}^{H_c}_{n+1-j}(\B_c, %H_c/\J_{j,c})=\operatorname{Tor}^{H'_c}_{n+1-j}(H'_c/\J'_{j,c},\B_c)$.
%Denote this bimodule by $\B_{j,c}$.
\item[(d)] We have $\J_{j-1,\lambda_1}\operatorname{Tor}^{\U_{\lambda_1}}_i(\B_{\lambda_1}, \U_{\lambda_1}/\J_{j,\lambda_1})=
\operatorname{Tor}^{\U_{\lambda_1}}_i(\U_{\lambda_1}/\J_{j,\lambda_1},\B_{\lambda_1})\J_{j-1,\lambda_1}=0$
for $i>n+1-j$.
\item[(e)]  Set $\B_{j,\lambda_1}:=\operatorname{Tor}^{\U_{\lambda_1}}_{n+1-j}(\B_{\lambda_1}, \U_{\lambda_1}/\J_{j,\lambda_1})$.
The kernel and the cokernel of the natural homomorphism $$\B_{j,\lambda_1}\otimes_{\U_{\lambda_1}}
\operatorname{Hom}_{\U_{\lambda_1}}(\B_{j,\lambda_1}, \U_{\lambda_1}/\J_{j,\lambda_1})\rightarrow
\U_{\lambda_1}/\J_{j,\lambda_1}$$
are annihilated by $\J_{j-1,\lambda_1}$ on the left and on the right.
\item[(f)] The kernel and the cokernel of the natural homomorphism
$$\operatorname{Hom}_{\U_{\lambda_1}}(\B_{j,\lambda_1}, \U_{\lambda_1}/\J_{j,\lambda_1})\otimes_{\U_{\lambda_1}}\B_{j,\lambda_1}
\rightarrow \U_{\lambda_1}/\J_{j,\lambda_1}$$ %as well as
%$$\operatorname{Tor}^{H_{c}}(\B_{c,j}, H_{c}/\J_{j,c})\otimes_{H'_c}\B_{c,j}
%\rightarrow H_{c}/\J_{j,c}$$
are annihilated on the left and on the right by $\J_{j-1,\lambda_1}$.
\end{itemize}
\end{Lem}
\begin{proof}
This lemma is an analog of \cite[Proposition 6.2]{rouq_der}. As in that proposition,
the proof is in four steps. First, we prove (a),(b) for a  Weil generic $\lambda_1$.
Then we  check the direct analogs of (c)-(f) for the ideal $\I_{\lambda_1}\subset \Walg$
(the minimal ideal of finite codimension) and the $\Walg_{\lambda_1}$-bimodule $\B_{\lambda_1,\dagger}$,
where $\lambda_1$ is Weil generic. In Step 3 we  establish (c)-(f) with an arbitrary $j$
and a Weil generic $\lambda_1$. Finally,  we will prove
the claims (a)-(f) for an arbitrary $j$ and a Zariski generic $\lambda_1$.

Step 1 is completely analogous to that of the proof of \cite[Proposition 6.2]{rouq_der}.
To prove Step 2, we first recall that $\B_{\lambda_1,\dagger}=\B^{\Walg}_{\lambda_1}$.
Also note that, by the argument in \cite[Section 4]{BL}, we have
$\operatorname{Tor}^{\Walg_{\lambda_1}}_i(\B^\Walg_{\lambda_1}, N)=0$ for any
$\Walg_{\lambda_1}/\I_{\lambda_1}$-module $N$ and $i\neq \frac{1}{2}\dim \tilde{S}$, and, moreover,
the functor $$N\mapsto \operatorname{Tor}^{\Walg_{\lambda_1}}_{\frac{1}{2}\dim \tilde{S}}(\B^{\Walg}_{\lambda_1},\bullet)$$    is a self-equivalence of $\Walg_{\lambda_1}/\I_{\lambda_1}\operatorname{-mod}$. Now the proof of Step 2
works in the same way as in \cite[Proposition 6.2]{rouq_der}.

The proofs of Steps 3 and 4 are the same as in {\it loc. cit.}
\end{proof}

The proof of Theorem \ref{Thm:perversity} now repeats that of \cite[Theorem 6.1]{rouq_der}.

Let us remark that the self-equivalence of $\Cat_{i}/\Cat_{i+1}$ induced by $\mathcal{F}(\lambda_1)$
is given by taking the tensor product with the bimodule $\B_{n+1-i,\lambda_1}$.

\subsection{Proof of (iii) of Theorem \ref{Thm:cactus}}
Note that, for a Zariski generic $\lambda_1\in \param_1$, the functor $\mathcal{F}(\lambda_1)$ coincides
with $\WC_{\lambda_1 w_{D_1}\leftarrow \lambda_1}$ provided $\lambda_1$ is regular and
dominant (note that the dominance is not a Zariski generic condition). When $\lambda_1$
is integral, we have an identification $\operatorname{Irr}(\OCat_{\lambda_1})\cong W$.
We write $\wc_{D_1}$ for the self-bijection of $W$ induced by the perverse equivalence
$\WC_{D_1}$.

Now we are going to prove (iii) of Theorem \ref{Thm:cactus}: the equality
$\wc_{D_1}(w'w'')=w'\wc^{(D_1)}_{D_1}(w'')$, where $\wc^{(D_1)}_{D_1}$ is the
analogous bijection for the category $\OCat(W_{D_1})$.

Recall that $\WC_{\lambda_1 w_{D_1}\leftarrow \lambda_1}$ is the inverse Ringel duality
provided $\lambda_1\in \param_1$ is Weil generic, Lemma \ref{Long_wc_perv}. It follows that $\WC_{\lambda_1 w_D\leftarrow \lambda_1}(\Delta(\mu))=
\nabla(\mu w_{D_1})$  for any $\mu\in W\lambda_1$. From here we deduce that $\mathcal{F}(\lambda_1)(\Delta(w\lambda_1))$
has no higher homology for a Zariski generic $\lambda_1\in \param_1$. A consequence of this
and Lemma \ref{Lem:WC_K_0} is that $\WC_{w_{D_1}}$ preserves the subcategories $D^b(\OCat(W)_{\prec c})\subset D^b(\OCat(W)_{\preceq c})$ for any right coset $c$ of $W_{D_2}$ with $D_1\subset D_2$.

Now let $c$ be a right coset for $W_{D_1}$.
We see that, for $w\in W_{D_1}$, the functors $\WC_{w}$ descend to the
the quotient $D^b(\OCat(W)_c)$ (this is because they are compositions of $\WC_{s_i}$,
where $s_i$ is a simple reflection in $W_{D_1}$). Recall, Lemma \ref{Lem:quot_parab},
that we have established the equivalence
$\OCat(W)_c\cong \OCat(W_1)$, where we write $W_1$ for $W_{D_1}$.
For $w\in W_1$, let us write $\WC_w$ for the functor of $D^b(\mathcal{O}(W)_c)$
induced by $\WC_w$ on $D^b(\mathcal{O}(W))$ and $\WC^1_w$ for the functor
on $D^b(\mathcal{O}(W_1))$.

\begin{Lem}\label{Lem:bij_coinc}
For any $w\in W_1$, there are abelian self-equivalences $\mathcal{F}_w,\mathcal{F}'_{w}$
of $\OCat(W_1)$ such that $\WC_w=\mathcal{F}_w\circ \WC^1_{w}=\WC^1_w\circ \mathcal{F}'_w$
and $\mathcal{F}_w,\mathcal{F}'_w$ induce the identity map on $K_0$.
\end{Lem}
\begin{proof}
The proof is by induction on the length $\ell(w)$. The base is $\ell(w)=1$, i.e.,
$w$ is a simple reflection $s_i, i\in D_1$. The endofunctor $\WC_{s_i}$ of
$D^b(\mathcal{O}(W))$ is a classical reflection functor, see
the proof of Lemma \ref{Lem:WC_K_0}. So there is only
one filtration term, it is spanned by $L_w$ with $ws_i<w$. If $ws_i<w$,
then $\WC_{s_i}L_w=L_w[-1]$. Also by the paragraph preceding the present lemma
it follows that $\wc_{i}$ preserves the right $\langle 1,s_i\rangle$-cosets.
It follows that the bijection induced by $\WC_{s_i}$ is trivial. From here
we deduce that the endofunctors $\WC_{s_i},\WC^1_{s_i}$ are perverse equivalences
with the same perversity data. Now we can use Lemma \ref{Lem:perv}
to establish the existence of $\mathcal{F}_{s_i},\mathcal{F}_{s_i'}$.

To prove the induction step we use the equalities $\WC_{s_i w}=\WC_{s_i}\circ\WC_w,
\WC^1_{s_i w}=\WC^1_{s_i}\circ \WC^1_w$ and Corollary \ref{Cor:perv_abel_precomp}.
\end{proof}

From Lemma \ref{Lem:bij_coinc}, we conclude that $\WC_{D_1},\WC^1_{D_1}$ are perverse
self-equivalences of $D^b(\OCat(W_1))$ with the same combinatorial data. Recall
that the equivalence $\OCat(W_1)\cong \OCat(W)_c$ sends $L_{w''}$ to $L_{w'w''}$.
This implies (iii) of the theorem.

\subsection{Cactus group action}\label{SS_cact_act}
Here we will verify that the self-bijections $\wc_{D_1}$ of $W$ satisfy the relations
of $\mathsf{Cact}_W$, see (\ref{eq:cactus_rel}).
\subsubsection{Relation $\wc_{D_1}^2=1$}
Recall from Lemma \ref{Lem:WC_K_0} that the functor $\WC_{D_1}$ acts on $K_0(\mathcal{O}(W))$
by the right multiplication by $w_D$ and so the action of
$\WC_{D_1}^2$ on the $K_0$ is trivial. Also note that, for $L_w\in \Cat_i\setminus \Cat_{i+1}$,
the image $\WC_{D_1}^2L_w$ equals $L_{w'}[-2i]$ modulo $D^b_{\Cat_{i+1}}(\OCat(W))$,
where $w'=\wc_{D_1}^2 w$. These two observations imply the claim.

\subsubsection{Relation $\wc_{D_1}\wc_{D_2}=\wc_{D_2}\wc_{D_1}$}
Thanks to (iii), we may assume that $D=D_1\sqcup D_2$.  Again, thanks to (iii), we have $\wc_{D_1}(w_1w_2)=\wc_{D_1}(w_1)w_2$ and $\wc_{D_2}(w_1w_2)=w_1\wc_{D_2}(w_2)$. Our claim follows.

\subsubsection{Relation $\wc_{D_1}\wc_{D_2}=\wc_{D_2}\wc_{D_1^*}$}
We can assume that $D_2=D$ thanks to (iii).

First of all, let us recall that
\begin{equation}\label{eq:W}
w_{D}^{-1}w_{D_1}w_D=w_{D_1^*}.
\end{equation}
%,  in $B_W$, we have $$T_{w_{D_1}}T_{w_{D_1}^{-1}w_D}=
%T_{w_D}=T_{w_D w_{D_1^*}^{-1}}T_{w_{D_1}^*}$$
%because $\ell(w_{D_1})\ell(w_{D_1}^{-1}w_D)=\ell(w_D)=\ell(w_D w_{D_1^*}^{-1})\ell(w_{D_1^*})$.
%But $w_{D_1^*}=w_{D}w_{D_1}w_{D}^{-1}$ and so $w_{D_1}^{-1}w_D=w_{D}w_{D_1^*}^{-1}$.
%It follows that
%\begin{equation}\label{eq:braid}
%T_{w_{D_1}}^{-1}T_{w_D}=T_{w_D}T_{w_{D_1^*}}^{-1}.
%\end{equation}
%Since the wall-crossing functors give a representation of $B_W$, we get
%\begin{equation}\label{eq:functor}
%\WC_{D_1^*}=\WC^{-1}_{D}\WC_{D_1}\WC_{D}.
%\end{equation}

Note that the functor $\WC_{D_1}$ preserves the filtration $D^b_{\leqslant i}(\OCat(W))
\subset D^b(\OCat(W))$ by the dimension of support. This is  because $\WC_{D_1}$ is 
given by taking the derived tensor product with a
HC bimodule. So we have an auto-equivalence $\WC_{D_1}$  of
$D^b_{\leqslant i}(\OCat(W))/D^b_{\leqslant i-1}(\OCat(W))$ that is perverse
with respect to the t-structure induced from $D^b(\OCat(W))$.
Note that the functor $\WC_D$ on  $D^b_{\leqslant i}(\OCat(W))/D^b_{\leqslant i-1}(\OCat(W))$
is t-exact up to a shift. By Lemma \ref{Lem:perv_compos_ab}, $\WC_D^{-1}\circ \WC_{D_1}\circ
\WC_D$ is a perverse self-equivalence of  $D^b_{\leqslant i}(\OCat(W))/D^b_{\leqslant i-1}(\OCat(W))$
such that the corresponding bijection is $\wc_{D}^{-1}\circ \wc_{D_1}\circ \wc_D$.
By (\ref{eq:W}) the actions of $\WC_D^{-1}\circ \WC_{D_1}\circ \WC_{D}$ and
$\WC_{D_1^*}$ on $K_0(\OCat(W)_i/\OCat(W)_{i-1})$ coincide. Applying
Lemma \ref{Lem:K_0_perv}, we finish the proof of $\wc_{D}^{-1}\circ \wc_{D_1}
\circ \wc_{D}=\wc_{D_1^*}$.

\subsection{Cells and Lusztig subgroups}
Here we will prove (i) and (ii) of Theorem \ref{Thm:cactus}.

Let $\lambda$ be a generic enough integral  element of $\param_1$
such that $\lambda w_{D_1}$ is dominant. Recall that  the subcategory
$\Cat^1_i\subset \OCat(W)$ consists
of all objects annihilated by an ideal $\J^1_{n+1-i}\subset \U_\lambda$, while the self-equivalence
$\Cat_i/\Cat_{i+1}$ is induced from taking the
tensor product with the bimodule
$\B_{n+1-i,\lambda_1}\in \HC^{1,1}_{\lambda,\lambda}(\U)$ mentioned in Lemma \ref{Lem:perv_bimod}.
 This bimodule is annihilated by $\J^1_{n+1-i}$ on the left and on
the right. Inside $\Cat_i/\Cat_{i+1}$ consider the Serre subcategory
$(\Cat_i/\Cat_{i+1})_{\overline{\Orb}}$ that is spanned by the images
of all simples in $\Cat_i$ whose associated variety is contained in $\overline{\Orb}$.
We can define $(\Cat_i/\Cat_{i+1})_{\partial\Orb}$ similarly.
So we can form the quotient $(\Cat_i/\Cat_{i+1})_{\Orb}$.
We note that, by the construction, the set $\operatorname{Pr}_{\Orb}(\U_\rho)$
splits into the union $\bigsqcup_i \operatorname{Pr}_{\Orb}(\U_\lambda)_i$
such that the simples in  $(\Cat_i/\Cat_{i+1})_{\Orb}$ are precisely those annihilated
by the primitive ideals in $\operatorname{Pr}_{\Orb}(\U_\lambda)_i$.

Since the equivalence $\Cat_i/\Cat_{i+1}\rightarrow \Cat_i/\Cat_{i+1}$
is given by taking a tensor product (from the left) with $\B_{n+1-i,\lambda_1}$,
it induces an equivalence $(\Cat_i/\Cat_{i+1})_{\Orb}\rightarrow
(\Cat_i/\Cat_{i+1})_{\Orb}$  by tensoring with
an object in the subquotient $\HC^{1,1}_{\lambda}(\U)_{\Orb}$. This self-equivalence
restricts to that of the  semisimple part
$(\Cat_i/\Cat_{i+1})_{\Orb}^{ss}\subset(\Cat_i/\Cat_{i+1})_{\Orb}$. This semisimple part is equivalent
(under the Beilinson-Bernstein equivalence) to the subcategory $\mathcal{E}_i \mathfrak{J}^{\Orb}$,
where $\mathcal{E}_i:=\sum_{\J}\U_\lambda/\J$, here the sum is taken over all
$\J\in \operatorname{Pr}_{\Orb}(\U_\lambda)_i$. From here we conclude that $\underline{\B}_i:=
\mathcal{B}_{n+1-i,\lambda_1}\otimes_{\U_\lambda}\mathcal{E}_i$ (the image of $\mathcal{E}_i$ under the equivalence)
is in $\mathcal{E}_i\mathfrak{J}^{\Orb}\mathcal{E}_i$. Tensoring from the left with this object
coincides with the self-equivalence of $(\Cat_i/\Cat_{i+1})_{\Orb}^{ss}$ of interest.

Recall the set $Y$ of the finite dimensional irreducible representations of $\Walg_\lambda$
that is equipped with an $\bar{A}$-action such that $Y/\bar{A}=\operatorname{Pr}_{\Orb}(\U_\lambda)$.
We can split $Y$ into the union $Y=\sqcup_i Y_i$ according
the decomposition of $\operatorname{Pr}_{\Orb}(\U_\lambda)$.
Then the category  $\mathcal{E}_i\mathfrak{J}^{\Orb}$ is identified with
$\operatorname{Sh}^{\bar{A}}(Y_i\times Y)$. The self-equivalence
of this category is given by convolving with
of $\underline{\mathcal{B}}_i$ viewed as an object in $\operatorname{Sh}^{\bar{A}}(Y_i\times Y_i)$.

\begin{Lem}\label{Lem:equi_convol}
Suppose the convolution with $\mathcal{F}\in \operatorname{Sh}^{\bar{A}}(Y_i\times Y_i)$ gives
an autoequivalence of $\operatorname{Sh}^{\bar{A}}(Y_i\times Y)$. Then the there is an
$\bar{A}$-equivariant bijection $\sigma:Y_i\rightarrow Y_i$ such that the graph of $\sigma$
coincides with the support of $\mathcal{F}$ and the fibers of $\mathcal{F}$ over the support
all have dimension $1$.
\end{Lem}
\begin{proof}
Let $\mathcal{F}'$ denote the inverse sheaf of $\mathcal{F}$. For any $x,y\in Y_i$, we have
$$\delta_{xy}=\sum_{z\in Y_i}\dim \mathcal{F}'_{x,z}\dim \mathcal{F}_{z,y}.$$
It follows that, for any $y\in Y_i$, there is exactly one $z$ such that
$\mathcal{F}_{z,y}\neq 0$ and in this case $\dim \mathcal{F}_{z,y}=1$.
The map $y\mapsto z$ is a required bijection. It is $\bar{A}$-equivariant because
$\mathcal{F}$ is $\bar{A}$-equivariant.
%Note that convolving with $\mathcal{F}$ cannot decrease  the maximal dimension of nonzero fibers.
%So the sheaves with all nonzero fibers of dimension $1$ will be mapped to sheaves with these
%properties. Also simple sheaves are mapped into simple sheaves.
%Let $O$ be an $\bar{A}$-orbit
%in $Y_i$ and let $\mathcal{E}_O$ be the sheaf with trivial 1-dimensional fiber over the diagonal
%in $O\times O$.
\end{proof}

The claim that $\wc_D$ preserves the right cells follows from the fact that $\wc_D$
is given by a convolution on the left. The claim that $\wc_D$ permutes the left
cells preserving the Lusztig subgroups follows from Lemma \ref{Lem:equi_convol}.
This finishes the proof of (i) of Theorem \ref{Thm:cactus}.

Let us proceed to the proof of (ii).
The construction above implies that the action of $\wc_{D_1}$ on a two-sided cell ${\bf c}$
is given by convolving with  $\mathcal{F}:=\bigoplus_i \underline{\B}_i\in \operatorname{Sh}^{\bar{A}}(Y\times Y)$
on the left.
Convolving with $\mathcal{F}^*$ on the right defines a commuting $\Cact_W$-action
on ${\bf c}$.
Clearly $(\mathcal{F}*\mathcal{G})^*=\mathcal{G}^**\mathcal{F}^*$. On the level of irreducible
objects the involution $\mathcal{G}\mapsto \mathcal{G}^*$ corresponds to
$w\mapsto w^{-1}$, see \cite[3.1 (g) and (h)]{Lusztig_leading}. This completes the proof of (ii).
%\subsection{Interpretation via the KL basis}
\subsection{Type A}
Here we are going to prove (iv) of Theorem \ref{Thm:cactus}. For this we will explain how to compute
$\wc_D$.

The following is a direct right handed analog of \cite[Theorem 3.1]{Mathas}.

\begin{Lem}\label{Lem:KL_invol}
There is a unique involution $\sigma:W\rightarrow W$ such that
$C_w T_{w_0}=\alpha C_{\sigma(w)}+\sum_{w'\prec_{LR}w}\beta_{w'}C_{w'}$
in $\mathcal{H}_v$, where $\alpha$ is of the form $\pm v^k$ for $k\in \Z$.
\end{Lem}

\begin{Lem}\label{Lem:cactus_KL}
We have $\wc_D=\sigma$.
\end{Lem}
\begin{proof}
Let $c_w\in \C W$ stand for the specialization of $C_w$ to $v=1$
so that $c_w=[L_w]$. We know that $[\WC_{D}L_w]\pm [L_{\wc_Dw}]\in
I^{LR}_{\prec w}$, where $I^{LR}_{\prec w}$ is the span of $c_{w'}$
with $w'\prec_{LR} w$. But $[\WC_D L_w]=[L_w]w_0$ by Lemma \ref{Lem:WC_K_0}.
By Lemma \ref{Lem:KL_invol}, $[L_w]w_0\pm c_{\sigma(w)}\in I^{LR}_{\prec w}$.
This implies the claim of the present lemma.
\end{proof}

Now let us assume that $W=\mathfrak{S}_n$. Recall that $\mathfrak{S}_n$
is in bijection (RSK) with the set of pairs $(P,Q)$ of standard Young
tableaux of the same shape. Let us write $(P_w,Q_w)$ for the tableaux
corresponding to $w$. Recall that $w\sim_L w'$ if and only if $Q_{w}=Q_{w'}$
and that $P_{w}=Q_{w^{-1}}, Q_w=P_{w^{-1}}$. Now it follows from
\cite[Section 3.6]{Mathas} that $Q_{\sigma(w)}=Q_{w}^*$, where $\bullet^*$
denotes the Sch\"{u}tzenberger involution uniquely specified by $Q_w^*:=Q_{w_0ww_0}$.

\begin{Ex}\label{Ex:n3}
Consider the case $n=3$. Here we have four  right cells: $\{\operatorname{id}\},
\{(12),(231)\}$, $\{(23),(312)\}$, $\{31\}$. The Sch\"{u}tzenberger involution
swaps the two standard Young tableaux of shape 21. So it swaps the elements
of the right cells with two elements.
\end{Ex}

The last example together with (iii) of Theorem \ref{Thm:cactus} shows that
the bijections $\wc_{i,i+1}$ are elementary Knuth transforms. It is known
that one can get any element of a given right cell from another element of
the same cell using a sequence of these transforms. This completes the proof
of (iv).

\section{Ramifications}\label{S_ramif}
\subsection{Unequal parameters and affine type}
An analog of Theorem \ref{Thm:cactus} should be true for Hecke algebras with unequal parameters
as well (though it is unclear whether Lusztig subgroups make sense in that generality).
Lusztig has proved a direct analog of Lemma \ref{Lem:KL_invol} in that setting,
see \cite{Lusztig_involution}. Then one can use (iii) of Theorem \ref{Thm:cactus}
to define the bijections $\mathfrak{wc}_{D_1}$ for an arbitrary subdiagram
$D_1\subset D$. The first two relations in (\ref{eq:cactus_rel}) come for
free and the last one should not be difficult to check.

The strategy of the first paragraph should work for $D$ of affine type as well.
The cactus group is defined in the same way (one  considers subdiagrams
$D_1$ of finite type only). In the equal parameter case, the action should have  a
categorical interpretation via affine categories $\mathcal{O}$, as in Section \ref{SS_cact_act}.

\subsection{Categories $\mathcal{O}$ for quantized symplectic resolutions
and Rational Cherednik algebras}
Our techniques should generalize to the case of categories $\mathcal{O}$ over quantizations
of symplectic resolutions, see \cite{BLPW}, and, perhaps, to categories $\mathcal{O}$
over Rational Cherednik algebras, \cite{GGOR}.

Let us elaborate on the symplectic resolution setting. For the  definitions and details a reader is
referred to \cite{BPW,BLPW}. An outcome is that we have two vector spaces,
the $\C$-space $\mathfrak{p}$ of quantization parameters with fixed rational
form $\mathfrak{p}_{\mathbb{Q}}$ and also a $\mathbb{Q}$-space $\mathfrak{s}$
(of rational co-characters of a certain  torus $T$). Both $\mathfrak{p}_{\mathbb{Q}}$
and $\mathfrak{s}$ come with a finite collection of codimension 1 subspaces
(walls) that split them into chambers and with integral lattices $\mathfrak{p}_{\Z},
\mathfrak{s}_{\Z}$.

For any $\theta\in \mathfrak{p}_{\mathbb{Q}}$ and $\nu\in
\mathfrak{s}_{\mathbb{Q}}$ that lie inside their chambers and any $\lambda\in \mathfrak{p}$,
we have a highest weight category $\mathcal{O}(\theta,\nu,\lambda)$. This category
depends only on the chambers of $\theta,\nu$ and $\lambda+\mathfrak{p}_{\Z}$. The simple
objects in these categories are labelled by the fixed points of $T$
on a certain symplectic variety $X$ (that is a symplectic resolution
of interest). For fixed $\lambda$, the categories $\OCat(\theta,\nu,\lambda)$
are derived equivalent: there are wall-crossing functors that switch $\theta$
that were introduced in \cite[Section 6.4]{BPW}, and cross-walling functors
introduced in \cite[Section 8.2]{BLPW}. The latter were  proved to be equivalences in
\cite[Theorem 6.3]{CW}.

The long wall-crossing functor (that switches the chamber of
$\theta$ to the opposite) was shown to be a perverse equivalence
in \cite[Section 4]{BL}. Using this and techniques of \cite[Section 6]{rouq_der}
one can show that wall-crossing functors through faces are perverse equivalences.
The corresponding bijections should satisfy the relations in the cactus groupoid
for the hyperplane arrangement in $\mathfrak{p}_{\mathbb{Q}}$.

The long cross-walling functor was shown in \cite{CW} to be an
inverse Ringel duality up to a homological shift. Since the inverse Ringel
duality coincides with the long wall-crossing, this allows to define
the long cross-walling bijection. To define these bijections for
arbitrary faces of the hyperplane arrangement in $\mathfrak{s}_{\mathbb{Q}}$
one can use \cite[Proposition 6.4]{CW}. Again, the resulting bijections
should satisfy cactus relations.

One can define the notions of left and right cells in $X^T$
(this depends on $\lambda$), see \cite[Section 7.5]{BLPW}.
By the definition there, the wall-crossing bijections
preserve the right cells, while the cross-walling bijections
should preserve the left cells.

The case considered in the present paper corresponds to $X=T^*(G/B)$,
$T$ being a maximal torus of $B$ and integral $\lambda$. In this case,
all categories $\mathcal{O}(\theta,\nu,\lambda)$ are naturally identified.
In a subsequent paper, we will consider in detail the case when
$X$ is a Nakajima quiver variety of affine type (the set $X^T$
has to do with multipartitions in this case).

Another setting, where one can define wall-crossing bijections,
is for categories $\mathcal{O}$ for Rational Cherednik algebras
$H_{1,c}(\Gamma)$, where $\Gamma$ is a complex reflection group.
The labelling set for the simple objects in this case is
$\operatorname{Irr}\Gamma$.
Wall-crossing functors (for a suitable hyperplane arrangement)
were introduced in \cite{rouq_der}. It was proved in that case that
wall-crossing through faces are perverse. The corresponding
wall-crossing bijections should be viewed as generalizations
and extensions of the Mullineux involution.

%Let $X_0$ be a normal affine Poisson
%variety equipped with a contracting $\C^\times$-action that rescales the Poisson bracket.
%By a symplectic resolution of $X_0$ one means a symplectic smooth algebraic variety $X$ that is
%equipped with a birational projective Poisson morphism $\pi:X\rightarrow X_0$. The action of
%$\C^\times$ uniquely lifts to $X$. We have an isomorphism $H^2(X,\mathbb{Q})\cong %\operatorname{Pic}(X)\otimes_{\Z}\mathbb{Q}$. There is Namikawa's Weyl group


\begin{thebibliography}{99}
\bibitem[BG]{BG} J. Bernstein, S. Gelfand. {\it Tensor products of finite and infinite dimensional representations of
semisimple Lie algebras}. Compositio Mathematica, 41(1980), n.2, p. 245-285.
\bibitem[BLu]{Bern_Lunts} J. Bernstein, V. Lunts, {\it Localization for derived categories of $(\g,K)$-modules}, J. Amer. Math. Soc. 8 (1995), no. 4, 819-856.
\bibitem[BFO]{BFO} R. Bezrukavnikov, M. Finkelberg, V. Ostrik, {\it Character D-modules via Drinfeld center of Harish-Chandra bimodules}. Invent. Math.  188  (2012),  no. 3, 589-620.
\bibitem[BLo]{BL} R. Bezrukavnikov, I. Losev. {\it Etingof conjecture for quantized quiver varieties}. arXiv:1309.1716.
\bibitem[BMR]{BMR} R. Bezrukavnikov, I. Mirkovic, D. Rumynin.
{\it Singular localization and intertwining functors for reductive Lie algebras in prime characteristic}.
Nagoya Math. J.  184  (2006), 1–55.
\bibitem[BR]{BR} C. Bonnafe, R. Rouquier. {\it Cellules de Calogero-Moser}. arXiv:1302.2720.
\bibitem[BLPW]{BLPW} T. Braden, A. Licata, N. Proudfoot, B. Webster, {\it Quantizations of conical symplectic resolutions II: category O and symplectic duality}. arXiv:1407.0964.
\bibitem[BPW]{BPW} T. Braden, N. Proudfoot, B. Webster, {\it Quantizations of conical symplectic resolutions I: local and global structure}. arXiv:1208.3863.
\bibitem[DJS]{DJS} M. Davis, T. Januskiewitz, R. Scott. {\it Fundamental groups of blow-ups}.
Adv. Math. 177 (2003), n.1, 115-179.
\bibitem[DK]{DK} C. Dodd, K. Kremnizer. {\it A Localization Theorem for Finite W-algebras}.
arXiv:0911.2210.
\bibitem[GG]{GG} W.L. Gan, V. Ginzburg. {\it Quantization of Slodowy slices}. IMRN, 5(2002), 243-255.
\bibitem[G]{Ginzburg_HC} V. Ginzburg. {\it Harish-Chandra bimodules for quantized Slodowy
slices}.  Represent. Theory, 13(2009), 236-271.
\bibitem[GGOR]{GGOR} V. Ginzburg, N. Guay, E. Opdam and R. Rouquier, {\it On the category $\mathcal{O}$ for rational
Cherednik algebras}, Invent. Math., {\bf 154} (2003), 617-651.
\bibitem[HK]{HK} A. Henriques, J. Kamnitzer. {\it Crystals and coboundary categories}. Duke Math. J.  132  (2006), 191-216.
\bibitem[L1]{HC} I.V. Losev. {\it Finite dimensional representations of
W-algebras}. Duke Math J. 159(2011), n.1, 99-143.
\bibitem[L2]{Gies} I. Losev. {\it Etingof conjecture for quantized quiver varieties II: affine quivers}.
arXiv:1405.4998.
\bibitem[L3]{rouq_der} I. Losev. {\it Derived equivalences for Rational Cherednik algebras}. arXiv:1406.7502.
\bibitem[L4]{B_ineq} I. Losev. {\it Bernstein inequality and holonomic modules}. arXiv:1501.01260.
\bibitem[L5]{CW} I. Losev. {\it On categories $\mathcal{O}$ for quantized symplectic resolutions}. arXiv:1502.00595.
\bibitem[LO]{LO} I. Losev, V. Ostrik, {\it Classification of finite dimensional irreducible modules
over W-algebras}. Compos. Math. 150(2014), N6, 1024-1076.
\bibitem[Lu1]{Lusztig_orange}  G. Lusztig. {\it Characters of reductive groups over a finite field}, Ann. Math. Studies 107, Princeton University Press (1984).
\bibitem[Lu2]{Lusztig_leading} G. Lusztig.  {\it Leading coefficients of character values of Hecke algebras}, Proc.Symp.Pure Math.47(2), Amer.Math.Soc. 1987, 235-262.
\bibitem[Lu3]{Lusztig_involution} G. Lusztig. {\it Action of longest element on a Hecke algebra cell module}.
arXiv:1406.0452.
\bibitem[Ma]{Mathas} A. Mathas. {\it On the left cell representations of Iwahori-Hecke algebras of finite Coxeter groups}.
J. London Math. Soc. (2)  54  (1996),  no. 3, 475-488.
\bibitem[Mi]{Milicic} D. Milicic.  {\it Localization and Representation Theory of Reductive Lie Groups}, available at
http://www.math.utah.edu/$\sim$milicic.
\bibitem[P1]{Premet1} A. Premet. {\it Special transverse slices and their enveloping algebras}. Adv. Math. 170(2002), 1-55.
\bibitem[P2]{Premet2} A. Premet. {\it Enveloping algebras of Slodowy slices and the Joseph ideal}.
J. Eur. Math. Soc, 9(2007), N3, 487-543.
\bibitem[R]{Rouquier_ICM} R. Rouquier. {\it Derived equivalences and finite dimensional algebras}.  International Congress of Mathematicians. Vol. II,  191-221, Eur. Math. Soc., Z\"{u}rich, 2006.
\bibitem[S]{Soergel} W. Soergel. \'{E}quivalences de certaines cat\'{e}gories de $\g$-modules, C. R. Acad.
Sci. Paris S\'{e}r. I Math. 303 (1986), no. 15, 725-728.
%\bibitem[S]{Stembridge} J. Stembridge. {\it Canonical bases and self-evaculating tableaux}.
%Duke Math. J. 82(1996), N3.
\end{thebibliography}
\end{document}